\documentclass[alphabetic]{article}

\usepackage{enumerate, longtable}
\usepackage[english]{babel}
\usepackage[autostyle, english=american]{csquotes}
\MakeOuterQuote{"}
\usepackage{amsmath, amscd, amsfonts, amsthm, amssymb, latexsym, comment, stmaryrd, graphicx, dsfont}
\usepackage{xcolor}
\usepackage[all]{xy}
\usepackage{fullpage}
\usepackage[
        colorlinks,
        linkcolor=red,  citecolor=blue,
        backref
]{hyperref}

\usepackage{amstext} 
\usepackage{array}
\usepackage{amsrefs}

\newtheorem{theorem}{Theorem}[section]

\newtheorem*{theorem*}{Theorem}
\newtheorem*{informal_theorem*}{Informal Theorem}

\newtheorem{lem}[theorem]{Lemma}

\newtheorem{corollary}[theorem]{Corollary}
\newtheorem{prop}[theorem]{Proposition}
\newtheorem{proposition}[theorem]{Proposition}

\theoremstyle{definition}

\newtheorem{rem}[theorem]{Remark}

\numberwithin{equation}{section}

\newcommand{\psiof}[1]{\psi\left(#1\right)}
\newcommand{\Wonefp}{W_1(f,\psi,\Phi)}
\newcommand{\Wone}[1]{W_1(#1, \psi, \Phi)}

\newcommand{\Wtwofp}{{W_2(f,\psi,\Phi)}}
\newcommand{\Pe}{\mathcal{P}}
\newcommand{\F}{\mathbb{F}}
\newcommand{\fqb}{{\overline{\F}_q}}
\renewcommand{\th}{\theta}
\newcommand{\Z}{\mathbb{Z}}
\newcommand{\R}{\mathbb{R}}
\newcommand{\AS}{\mathcal{AS}}
\newcommand{\g}{\mathfrak{g}}
\newcommand{\ord}{\mathrm{ord}}

\newcommand{\odd}{\mathrm{odd}}
\newcommand{\PP}{\mathbb{P}}

\newcommand{\floor}[1]{{\left\lfloor #1\right\rfloor}}

\newcommand{\case}[1]{\vspace{0.3cm}\noindent\framebox{#1.}}
\newcommand{\maincase}[1]{{\setlength{\fboxrule}{1.5pt}\vspace{0.3cm}\noindent\framebox{#1.}}}

\makeatother

\usepackage{babel}

\begin{document}
\global\long\def\field{\mathbb{F}}

\global\long\def\remark{\textbf{Remark.}}

\global\long\def\Scal{\mathcal{S}}
\global\long\def\ASdg{\AS_{d,g}^{\mathrm{ord}}}

\global\long\def\vol{\mathrm{vol}}

\global\long\def\pr{\mathrm{pr}}
\global\long\def\Lcp{L(u,C,\psi)}

\global\long\def\tr{\mathrm{tr}}
\global\long\def\USp{\mathrm{USp}}
\global\long\def\hd{\mathcal{H}_d}
\global\long\def\prgrp#1{{#1^\pr}}

\global\long\def\hg{\mathcal{H}_g}
\global\long\def\d{\mathrm{d}}
\global\long\def\etarhq#1{\eta(#1;H_Q^\perp)}

\global\long\def\vonmongchi#1{\sum_{\deg(h)=#1}\Lambda(h)\chi(h)}

\global\long\def\mhg{M_{\hg}}
\global\long\def\mean#1{\left\langle#1\right\rangle}
\global\long\def\tmean#1{\langle#1\rangle}
\global\long\def\P{\mathcal{P}}

\global\long\def\Gal{\mathrm{Gal}}

\global\long\def\Z{\mathbb{Z}}
\global\long\def\roots{\mathrm{roots}}

\global\long\def\asfam{{\AS_d^0}}
\global\long\def\ASord{\AS_d^{\mathrm{ord}}}
\global\long\def\oasfam{\AS_d^{0,\odd}}
\global\long\def\Lfp{L(u,f,\psi)}
\global\long\def\LCf{L(u,C_f)}
\global\long\def\R{\mathbb{R}}
\global\long\def\reals{\mathbb{R}}
\global\long\def\N{\mathbb{N}}
\global\long\def\C{\mathbb{C}}
\global\long\def\Tfr#1{{T_{f,\psi}^{#1}}}
\global\long\def\H{\mathcal{H}}

\title{Local Statistics for Zeros of Artin-Schreier $L$-functions}
\author{Alexei Entin and Noam Pirani}
\date{}
\maketitle

\begin{abstract}

We study the local statistics of zeros of $L$-functions attached to Artin-Scheier curves over finite fields. We consider three families of Artin-Schreier $L$-functions: the ordinary, polynomial (the $p$-rank 0 stratum) and odd-polynomial families. We compute the 1-level zero-density of the first and third families and the 2-level density of the second family for test functions with Fourier transform supported in a suitable interval. In each case we obtain agreement with a unitary or symplectic random matrix model.

\end{abstract}

\section{Introduction}

According to the Katz-Sarnak philosophy \cite{KaSa99a} (extending observations of Montgomery \cite{Mon73} and Odlyzko \cite{Odl87} regarding the Riemann Zeta-function) with every natural family of $L$-functions one may associate a random matrix model such that statistically the zeros of a random $L$-function drawn from the family should behave like the eigenvalues of a random matrix drawn from a corresponding classical ensemble (which is always unitary, symplectic or orthogonal\footnote{For some families one needs to discard certain fixed zeros at the critical point for this to hold, see \cite{Mil06,DHP15}.}). In the present work we will obtain further evidence that this philosophy applies in the case of three natural families of Artin-Schreier $L$-functions by studying the \emph{local statistics} of their zeros. We consider the ordinary, polynomial (or $p$-rank 0) and odd polynomial families and establish partial results consistent with a unitary (ordinary and polynomial cases) or symplectic (odd polynomial case) random matrix model.

Previously the local zero statistics for the polynomial Artin-Schreier family were studied by the first author \cite{Ent12}. The \emph{mesoscopic} statistics for the ordinary and polynomial families (as well as other $p$-rank strata) were studied by Bucur, David, Feigon, Lal\'in and Sinha in \cites{BDFLS12,BDFL16}, however the local regime is usually more challenging than the mesoscopic one. In the present work we improve the results of \cite{Ent12} in the polynomial case and establish results on local statistics for the ordinary and odd polynomial cases for the first time (the latter case is particularly interesting because it involves symplectic statistics). We note that all of our results apply in the fixed finite field regime. In the large finite field limit (with degrees slowly going to infinity) at least our first two results follow (and without restriction on the test functions) from deep equidistribution results of Katz \cite[Theorems 3.9.2 and 3.10.7]{Kat05} building on the work of Katz and Sarnak \cite{KaSa99}. 
Other families of $L$-functions over function fields for which local statistics have been studied in the fixed $q$ regime include hyperelliptic \cite{Rud10,Rod12,ERR13,BuFl18,BaJu19}, Dirichlet \cite{AMPT14}, cyclic $\ell$-covers and noncyclic cubic covers of $\mathbb P^1_{\F_q}$ \cite{BCDGL18, Mei22} and twists of elliptic curves over $\F_q(t)$ \cite{Com22}.
The corresponding literature over number fields is more extensive, see e.g. \cite{DaGu21}*{p. 2} for a list of references.

Let $p>2$ be a prime, $q$ a power of $p$ and $\F_q$ the field with $q$ elements. An Artin-Schreier (henceforth abbreviated as A-S) curve over $\F_q$ is given by the affine equation $y^p-y=f(x)$, where $f\in\F_q(x)\setminus\F_q$ is a nonconstant rational function which cannot be written as $f=h^p-h$ with $h\in\F_q(x)$. Denote by $C_f$ the smooth projective model of the affine curve $y^p-y=f(x)$. The curve $C_f$ is always absolutely irreducible and comes equipped with a cover $C_f\to\mathbb P^1_{\F_q}$ given by the projection $(x,y)\mapsto x$. We call such a cover an A-S cover. If a curve $C/\F_q$ is isomorphic to $C_f$ with $f\in\F_q[x]$ a polynomial we say that $C$ is a polynomial A-S curve.

For a (smooth, projective, irreducible) curve $C$ we denote by $\g(C)$ its genus. Another important invariant associated with a curve $C/\F_q$ is its \emph{$p$-rank}, which is defined to be the $\Z/p$-rank of $\mathrm{Jac}(C_f\times\overline{\F}_q)[p]$. It is a nonnegative integer $\le\g(C)$. A curve with $p$-rank $\g(C)$ is called \emph{ordinary}. An A-S curve is polynomial (in the sense of the previous paragraph) iff it has $p$-rank 0. For background on Artin-Schreier curves and their $p$-rank stratification see \cite{PrZh12}.

The zeta function of $C_f$ is defined by
\begin{equation}
Z_{C_f}(u)=\exp\left[{\sum_{r=1}^\infty}N_r(C_f)\frac{u^r}{r}\right],
\end{equation}
where $N_r(C_f)=\#C_f(\F_{q^r})$ is the number of projective $\F_{q^r}$-points of $C_f$,
and it is known that 
\begin{equation*}
Z_{C_f}(u)=\frac{L(u,C_f)}{(1-u)(1-qu)},
\end{equation*}
where $L(u,C_f)$, called the $L$-function of $C_f$, is a polynomial of degree $2\g(C_f)$. From this it follows that
$$
L(u,C_f)=\exp\left[\sum_{r=1}^{\infty}\frac{N_r(C_f)-q^r-1}{r}u^r\right].
$$

Denote by $\Psi$ the set of non-trivial additive characters of $\field_{p}$, and denote by $\tr_{q/p}$ the trace map $\tr_{\field_q/\field_p}$ (for a $p$-power $q$). Note that $|\Psi|=p-1$. For an A-S curve $C_f$ it is known that $L(u,C_f)$ factors into primitive $L$-functions as follows: 
\begin{equation}\label{eq:LuCf}
L(u,C_f)=\prod_{\psi\in\Psi} L(u,f,\psi),
\end{equation}
with
\begin{equation}\label{eqn_LfpsiDef}
L(u,f,\psi)=\exp\left[\sum_{r=1}^\infty \sum_{\alpha\in\field_{q^r}\cup\{\infty\}\atop{f(\alpha)\neq\infty}}\psi(\tr_{q^r/p}f(\alpha))\frac{u^r}{r}\right]
\end{equation}
(with the value $f(\alpha)$ defined by viewing the rational function $f$ as a function $f:\mathbb P^1(\overline\F_q)\to\mathbb P^1(\overline\F_q)$; for the derivation of (\ref{eq:LuCf}),(\ref{eqn_LfpsiDef}) see \cite{BDFL16}*{\S 2}).
Each of the factors in \eqref{eqn_LfpsiDef} is a polynomial of degree $\delta=2\g(C_f)/(p-1)$ (in particular $\delta$ is always an integer). Thus, studying the zeros of $L(u,C_f)$ reduces to studying the zeros of $L(u,f,\psi)$ for each $\psi$.

For the rest of the section fix $\psi\in\Psi$ (the choice of $\psi$ is unimportant). The Riemann Hypothesis for curves over finite fields guarantees that the roots of $L(u,f,\psi)$ have absolute value $q^{-1/2}$.
We denote the roots (with multiplicity) by $\lambda_{i}(f,\psi)$, $i=1,\ldots,2\g(C_f)/(p-1)$, and we define $\rho_i(f,\psi)=q^{-1/2}\lambda_{i}^{-1}(f,\psi)$, which are called the normalized (inverse) zeros of $L(u,f,\psi)$. Here $\rho_i\in\C$ lie on the unit circle. We denote $\th_i(f,\psi)=\arg(\rho_i(f,\psi))/2\pi\in\R/\Z$.

We will consider three families of Artin-Schreier curves (and their corresponding $L$-functions), all paramet\-rized by the degree $d=\deg f\ge 1$ of the rational function $f$ defining the curve $C_f$ via $y^p-y=f(x)$.

\begin{enumerate}
\item The \emph{Ordinary} Artin-Schreier family:
\begin{multline*}\AS_d^\ord=\big\{f=h/g:h,g\in\F_q[t],g\mbox{ monic squarefree},(h,g)=1,\\
(\deg g=d,\deg h\le d)\mbox{ or } (\deg g=d-1,\deg h=d)\big\}.\end{multline*}
Each curve $C_f$ with $f\in\AS_d^\ord$ has genus $\g=(p-1)(d-1)$ and is ordinary in the sense of having $p$-rank $\g(C_f)$. Moreover, each isomorphism class of ordinary A-S covers $C_f\to\PP^1_{\F_q}$ (defined over $\F_q$) of genus $\g$ appears exactly $q/p$ times in this family (the reason for the multiplicity is that replacing $f$ by $f+b$ where $b\in\F_q,\tr_{q/p}=0$ gives an isomorphic cover).

For any monic squarefree $g\in\F_q[x]$ of degree $d$ or $d-1$ we define the subfamily with fixed denominator
$$\AS_{d,g}^\ord=\{f=h/g:h\in\F_q[x],(h,g)=1,f\in\AS_d^\ord\}.$$

\item The \emph{Polynomial} Artin-Schreier family, defined for $(d,p)=1$:
$$\AS_d^0 = \left\{f\in\F_q[x]:f=\sum_{i=0}^da_ix^i, a_d\neq 0, a_i=0\mbox{ if }i>0,p|i\right\}.$$
Each curve $C_f$ with $f\in\AS_d^0$ has genus $\g=(p-1)(d-1)/2$ and $p$-rank 0, and each isomorphism class of covers $C_f\to\PP^1_{\F_q}$ with genus $\g$ and $p$-rank 0 appears exactly $q/p$ times in this family (the reason for omitting exponents divisible by $p$ is to avoid further multiplicity).

\item The \emph{Odd polynomial} Artin-Schreier family, defined for $(d,2p)=1$:
$$\AS_d^{0,\odd} = \left\{f\in\AS_d^0:f(x)=-f(-x)\right\}.$$
The curves in this subfamily of the polynomial A-S family have the property that $L(u,f,\psi)\in\R[u]$ for each $f\in\AS_d^{0,\odd}$ (and the leading coefficient of $L(u,f,\psi)$ is $q^{(d-1)/2}$, which is positive) and therefore the $d-1$ roots $\rho_i(f,\psi)$ come in conjugate pairs. Once again each isomorphism class of A-S covers in this family appears exactly $q/p$ times.

\end{enumerate}

The distribution of $L$-zeros for the ordinary and polynomial A-S families will be modelled by the eigenvalues of random unitary matrices. For the odd polynomial family it will be modelled by random unitary symplectic matrices.

Next we define the 1-level and 2-level densities for the zeros of $L(u,f,\psi)$, which measure the local statistics of zeros near the critical point $u=q^{-1/2}$. Let $\Phi\in\mathcal{S}(\R)$ ($\mathcal{S}(\R)$ denotes the space of Schwartz functions on $\R$) be a test function. The 1-level density of $\Lfp$ with test function $\Phi$ is defined by
$$W_1(f,\psi,\Phi)=\sum_{i=1}^{2\g/(p-1)}\phi_{2\g/(p-1)}(\th_i(f,\psi)),$$
where \begin{equation}\label{eq:periodic}\phi_{2\g/(p-1)}(t)=\sum_{n\in\Z}\Phi\left(\frac{2\g}{p-1}(t+n)\right),\g=\g(C_f),\end{equation} is the periodic sampling function (with period 1) associated with $\Phi$ at the scale $(p-1)/2\g$ (which is the natural local scale since there are $2\g/(p-1)$ zeros).
Similarly for a bivariate test function $\Phi\in\mathcal{S}(\R^2)$ we define
$$W_2(f,\psi,\Phi)=\sum_{i,j=1\atop{i\neq j}}^{2\g/(p-1)}\phi_{2\g/(p-1)}(\th_i(f,\psi),\th_j(f,\psi)),$$
where $\phi_{2\g/(p-1)}(t,s)=\sum_{m,n\in\Z}\Phi\left(\frac{2\g}{p-1} (t+m), \frac{2\g}{p-1} (s+n)\right)$.

Similarly, one defines the 1-level and 2-level densities for the eigenvalues of a unitary matrix: if $U\in\mathrm{U}_N$ is an $N\times N$ unitary matrix with eigenvalues $e^{2\pi i\theta_j(U)},j=1,\ldots,N$ and $\Phi\in\mathcal{S}(\R)$ (resp. $\Phi\in\mathcal{S}(\R^2)$) we define
$$W_1(U,\Phi)=\sum_{i=1}^N\phi_{N}(\th_i(U)),$$
$$W_2(U,\Phi)=\sum_{i,j=1\atop{i\neq j}}^N\phi_{N}(\th_i(U),\th_j(U)),$$
(here $\phi_N$ is defined as above but with scaling factor $N$).
For a random variable $X$ on a probability space $\Omega$ we denote by $\mean{X(a)}_{a\in\Omega}$ its mean (expected value). If $\Omega$ is finite we will always equip it with the uniform probability measure. The asymptotics of the mean of 1-level and 2-level density (more generally of $n$-level density) for a random matrix drawn from a classical compact group as $N\to\infty$ are known \cite{KaSa99}*{Theorem AD.2.2}. In particular one has
\begin{equation}\label{1levelU}\lim_{N\to\infty}\mean{W_1(U,\Phi)}_{U\in\mathrm{U}(N)}=\int_{-\infty}^\infty\Phi(t)\mathrm{d}t,\end{equation} 
\begin{equation}\label{2levelU}\lim_{N\to\infty}\mean{W_2(U,\Phi)}_{U\in\mathrm{U}(N)}=\iint_{\R^2}\Phi(t,s)\left(1-\left(\frac{\sin(\pi(t-s))}{\pi(t-s)}\right)^2\right)\mathrm{d}t\mathrm{d}s,\end{equation}
\begin{equation}\label{1levelUSp}\lim_{N\to\infty}\mean{W_1(U,\Phi)}_{U\in\mathrm{USp}(2N)}=\int_{-\infty}^\infty\Phi(t)\left(1-\frac{\sin 2\pi t}{2\pi t}\right)\mathrm{d}t,\end{equation}
where the means are taken with respect to the Haar probability measure. The main results of the present paper establish analogous limits for the zeros of $L$-functions chosen randomly (with uniform probability) from the above three Artin-Schreier families (for a restricted class of test functions), thus providing new evidence for the random matrix models for these families.

Throughout the paper we use the following definition of the Fourier transform:
$$\hat\Phi(\tau_1,\ldots,\tau_n)=\int_{\R^n}\Phi(t_1,\ldots,t_n)e^{-2\pi i\sum_{j=1}^n\tau_jt_j}\mathrm{d}t_1\cdots\mathrm{d}t_n.$$

\begin{theorem}[2-level density of polynomial family]\label{thm:main1}
Assume $p>2$. Let $\Phi\in\mathcal{S}(\reals^{2})$ be a fixed test function with with Fourier transform $\hat\Phi(\tau,\sigma)$ supported on the region $|\tau|+|\sigma|<2-2/p$. Then (compare with (\ref{2levelU}))
$$\lim_{d\to\infty}\mean {W_2(f,\psi,\Phi)}_{f\in\AS_d^0}=\iint_{\R^2}\Phi(t,s)\left(1-\left(\frac{\sin(\pi(t-s))}{\pi(t-s)}\right)^2\right)\mathrm{d}t\mathrm{d}s,$$
uniformly in $q$.  
\end{theorem}

\begin{theorem}[1-level density of odd polynomial family]\label{thm:main2} Assume $p>2$. Let $\Phi\in\mathcal{S}(\reals)$ be a fixed test function with with Fourier transform $\hat\Phi(\tau)$ supported on the interval $|\tau|<1-1/p$. Then (compare with (\ref{1levelUSp}))
$$\lim_{d\to\infty}\mean {W_1(f,\psi,\Phi)}_{f\in\AS_d^{0,\odd}}=\int_{-\infty}^\infty\Phi(t)\left(1-\frac{\sin 2\pi t}{2\pi t}\right)\mathrm{d}t,$$
uniformly in $q$.
\end{theorem}

\begin{theorem}[1-level density of ordinary family with fixed denominator]\label{thm:main3} Assume $p>2$. Let $\Phi\in\mathcal{S}(\reals)$ be a fixed test function with with Fourier transform $\hat\Phi(\tau)$ supported on the interval $|\tau|<1$. Then  we have (compare with (\ref{1levelU}))
$$\lim_{d\to\infty}\mean {W_1(f,\psi,\Phi)}_{f\in\AS_{d,g}^\ord}=\int_{-\infty}^\infty\Phi(t)\mathrm{d}t$$
uniformly over $q$ and monic squarefree $g\in\F_q[x]$ with $\deg g\in\{d,d-1\}$.
\end{theorem}

By averaging over all suitable $g$ we obtain the following

\begin{corollary}[1-level density of full ordinary family]\label{cor:main3} Under the conditions of Theorem \ref{thm:main3} we have
$$\lim_{d\to\infty}\mean {W_1(f,\psi,\Phi)}_{f\in\AS_{d}^\ord}=\int_{-\infty}^\infty\Phi(t)\mathrm{d}t,$$ uniformly in $q$.\end{corollary}

\begin{rem} In \cite{Ent12} the first author established Theorem \ref{thm:main1} with the smaller range $|\tau|+|\sigma|<1$ (see \cite{Ent12}*{Theorem 3}\footnote{In \cite{Ent12}*{Theorem 3} the integral on the right hand side is slightly miswritten, the correct kernel (which is what the original proof produces) is the one appearing in Theorem \ref{thm:main1} of the present paper (up to rescaling by $2\pi$).}) as well as the 1-level density for the polynomial A-S family in the range $|\tau|<2-2/p$ (see \cite{Ent12}*{Theorem 2}) 
Theorems \ref{thm:main2} and \ref{thm:main3} are new for any nonzero test function.\end{rem}

\begin{rem} We conjecture that in all three cases the results hold for any Schwartz test function. However in all cases where results of this form have been proved for families $\mathcal{F}_d$ of $L$-functions in the fixed $q$ regime, the Fourier transform of the test function was required to be supported in the range up to $2\lim_{d\to\infty}\log_q|\mathcal{F}_d|/c(\mathcal{F}_d)$ or smaller, where $c(\mathcal{F}_d)$ is the analytic conductor of the family (defined to be the number of zeros of an $L$-function in the family). This can be considered the limit of current technology (a similar observation applies over number fields). Theorems \ref{thm:main1},\ref{thm:main2} and \ref{thm:main3} reach that barrier. In the case of Corollary \ref{cor:main3} we have not been able to exploit the averaging over $g$ to increase the range. See \cite{DPR20_} for a striking recent result that does make use of external averaging to slightly increase the range in a similar situation. \end{rem}

\begin{rem} Here we only consider 1- and 2-level densities of \emph{low-lying} zeros, i.e. zeros near $\theta=0$. We could easily consider the local statistics of zeros around any given value of $\theta\in\R/\Z$ defined by shifting the arguments in $\phi_{2\g/(p-1)}$ above. A slight modification of our calculations would give the corresponding results. By averaging over $\theta$ we can then show that the pair correlation of zeros for the polynomial A-S family is the same as for random unitary matrices as long as the test function is supported in $(-(1-1/p),1-1/p)$.\end{rem}

In proving all of our results we take the standard approach of applying the explicit formula to reduce the problem to the estimation of certain sums over primes, which are then separated into diagonal and off-diagonal terms. One of the technical novelties in the present work, which is instrumental in all of our results, is using bounds on the number of short vectors in $\F_q[x]$-lattices to bound the off-diagonal terms. Here Lenstra's theory of reduced bases for such lattices \cite{Len85} plays a pivotal role. This approach exploits unique features of the A-S families which are not present in other well-studied families of $L$-functions.

The paper is organized as follows: in section \ref{sec:dirichlet} we relate the above families of A-S $L$-functions with certain natural families of Dirichlet $L$-functions. Most of our calculations proceed in the Dirichlet $L$-function setting which turns out to be convenient for our purposes. In section \ref{sec:lattice} we recall the basic theory of reduced bases for $\F_q[x]$-lattices developed in \cite{Len85} and derive several bounds on short vectors in such lattices, which will be used in our estimates of the off-diagonal terms. With these preliminaries in place, we prove Theorems \ref{thm:main1},\ref{thm:main2},\ref{thm:main3} in sections \ref{sec_full_as_family},\ref{sec_odd_poly},\ref{sec:ord_as_fam} respectively. The latter three sections are mostly independent of each other and can be read in any order.
\\ \\
{\bf Acknowledgments.} The authors would like to thank Ze\'ev Rudnick for providing some useful references on counting short vectors in lattices. Both authors were partially supported by a grant of the Israel Science Foundation no.~2507/19. 

\section{Reformulation in terms of Dirichlet $L$-functions}\label{alexei_construction_section}
\label{sec:dirichlet}

Throughout the rest of the paper $p>2$ is a prime, $q$ is a power of $p$ and we fix a non-trivial additive character $\psi:\F_p\to\C^\times$. In the present section we explain how to associate a primitive Dirichlet character $\chi_f$ with each A-S $L$-function $L(u,f,\psi)$ in each of the families we consider, so that $L(u,\chi_f)=(1-u)L(u,f,\psi)$ (the left hand side is a Dirichlet $L$-function). That a suitable (generalized) Dirichlet character exists for any A-S $L$-function follows from class field theory, but we will need to compute these characters explicitly and will only do so for polynomial and ordinary A-S $L$-functions.

\subsection{Dirichlet characters and L-functions} \label{subsec:Dir}
We recall some basic facts about Dirichlet characters over $\F_q[x]$ and their associated $L$-functions. See \cite{Ros02}*{Chapter 4} for a thorough introduction. 

Denote by $\mathcal{M}$ the set of monic polynomials in $\field_q[x]$. Let $Q(x)\in\mathcal{M}$ be a monic polynomial of degree $m\ge 1$, and let $\chi$ be a character of the multiplicative group $(\field_q[x]/Q)^{\times}$. We may extend $\chi$ to a strongly multiplicative function $\chi:\F_q[x]\to\C$ defined by

\[
    \chi(g)= 
\begin{cases}
    0,& (g,Q)\ne 1,\\
    \chi(g\text{ mod }Q),              & \mbox{otherwise}.
\end{cases}
\]
The extended function is called a (classical) Dirichlet character. A Dirichlet character $\chi$ modulo $Q$ is called \emph{primitive} if there is no proper divisor $Q_1|Q$ such that for $g\in\F_q[x]$ coprime with $Q$ the value $\chi(g)$ depends only on $g\text{ mod }Q_1$. It is called \emph{trivial} if it takes $1$ on every polynomial coprime with $Q$. It is called \emph{even} if it takes the value $1$ on the constants $\field_q^{\times}$, and \emph{odd} otherwise. 
We denote
$$e(\chi)=\left\{\begin{array}{ll} 1,&\chi\mbox{ is even},\\ 0,&\chi\mbox{ is odd}.\end{array}\right.$$

If $Q\in\mathcal M$ is a modulus we denote by $(\F_q[x]/Q)^{\times*}$ the group of Dirichlet characters modulo $Q$ (which we identify with the dual group of $(\F_q[x]/Q)^{\times}$). If $H\subset (\F_q[x]/Q)^{\times*}$ is a subset, we denote by $H^\pr$ the set of primitive characters in $H$.

To a Dirichlet character $\chi$ one attaches its Dirichlet $L$-function
\begin{equation}\label{eq:def_dir_l}
L(u,\chi)=\sum_{F\in\mathcal{M}} \chi(F)u^{\deg{F}} =\prod_{P\in\mathcal{P}}(1-\chi(P)u^{\deg{P}})^{-1}=\sum_{r=0}^\infty A_r(\chi)u^r,
\end{equation}
where $A_r(\chi)=\sum_{F\in\mathcal{M}\atop{\deg F=r}}\chi(F)$ and $\mathcal{P}$ is the set of prime polynomials in $\field_q[x]$, where by a prime polynomial we always mean a monic irreducible polynomial. 
If $\chi$ is non-trivial then $A_r(\chi)=0$ for $r\ge m=\deg Q$ and thus $L(u,\chi)$ is a polynomial of degree $\le m-1$ in the variable $u$.

If $\chi$ is a primitive character, $L(u,\chi)$ factors as follows:
$$
L(u,\chi)=(1-u)^{e(\chi)}\prod_{i=1}^{m-1-e(\chi)}(1-\rho_iq^{1/2}u).
$$
By the Riemann Hypothesis for curves over finite fields we also have $|\rho_i|=1$, and $\rho_i$ are called the normalized (inverse) roots or zeros of $L(u, \chi)$.

\subsection{Associating a Dirichlet character with a polynomial A-S curve}
\label{sec:dir_pol}

In the present subsection we recall the content of \cite{Ent12}*{\S 7}, where the Dirichlet character associated with a polynomial A-S $L$-function and its basic properties were studied.
Recall that we fix a non-trivial additive character $\psi:\F_p\to\C^\times$. Let $d\in\N$ be coprime with $p$ and define the following subfamily of $\AS_d^0$:
\begin{equation}\label{eq:fd}\mathcal F_d=\{f\in\AS_d^0:f(0)=0\}=\left\{f=\sum_{i=1}^da_ix^i\in\F_q[x]:a_i=0\mbox{ if }i|p\right\}.\end{equation}
We have \begin{equation}\label{eq:as_fd}\AS_d^0=\bigsqcup_{b\in\F_q}\{f+b:f\in\mathcal F_d\}\end{equation}
(disjoint union).

\begin{lem}\label{lem:twist} Let $L(u,f,\psi)$ be an A-S $L$-function (here we only assume that $f\in\F_q(x)\setminus\F_q$ is not of the form $f=h^p-h,h\in\F_q(x)$) with normalized (inverse) roots $\rho_1,\ldots,\rho_m,m=2\g(C_f)/(p-1)$. Then $L(u,f+b,\psi)=L(\psi(\tr_{q/p}b)\cdot u,f,\psi)$.\end{lem}

\begin{proof} By (\ref{eqn_LfpsiDef}),
\begin{multline*}L(u,f+b,\psi)=\exp\left[\sum_{r=1}^\infty \sum_{\alpha\in\field_{q^r}\cup\{\infty\}\atop{f(\alpha)\neq\infty}}\psi(\tr_{q^r/p}(f(\alpha)+b))\frac{u^r}{r}\right]\\=
\exp\left[\sum_{r=1}^\infty \sum_{\alpha\in\field_{q^r}\cup\{\infty\}\atop{f(\alpha)\neq\infty}}\psi(\tr_{q^r/p}f(\alpha)+r\cdot\tr_{q/p}b)\frac{u^r}{r}\right]
 =\exp\left[\sum_{r=1}^\infty \sum_{\alpha\in\field_{q^r}\cup\{\infty\}\atop{f(\alpha)\neq\infty}}\psi(\tr_{q^r/p}f(\alpha))\psi(\tr_{q/p} b)^r\frac{u^r}{r}\right]
\\=L(\psi(\tr_{q/p}b)u,f,\psi).\end{multline*}

\end{proof}

Lemma \ref{lem:twist} combined with (\ref{eq:as_fd}) shows that studying $L$-functions in the 
family $\AS_d^0$ essentially reduces to studying $L$-functions in the family $\mathcal F_d$. We 
will mostly work with $\mathcal F_d$ and only in the end deduce results for $\AS_d^0$. The main 
reason working with $\mathcal F_d$ is preferable is that it can be related to a family of 
Dirichlet $L$-functions. Let $f\in\mathcal F_d^0$ be a polynomial. We now define a function $\chi_f:
\F_q[x]\to\C^\times$ which will turn out to be a primitive Dirichlet character modulo $x^{d+1}$.

Let $c\in\field_q[x]$ be a polynomial of degree $l$. If $x|c$ we put $\chi_f(c)=0$. Otherwise we may factor $c(x)$ over the algebraic closure of $\field_q$:
$$
c(x)=c(0)\prod_{i=1}^{l}(1-\alpha_i x),
$$
and define
\begin{equation}\label{eqn_chifDef}
    \chi_f(c)=\psiof{\tr_{q/p}\sum_{i=1}^l f(\alpha_i)},
\end{equation}
where we recall that $\tr_{q/p}$ denotes the trace map from $\field_q$ to $\field_p$. To show that the right hand side in \eqref{eqn_chifDef} is well-defined we note that $\sum_{i=1}^l f(\alpha_i)\in\field_q$ since it is a symmetric function in the roots of $c$ with coefficients in $\F_q$. 

\begin{proposition}\label{prop:astodir} Assume $(d,p)=1$. 
\begin{enumerate}
\item[(i)] For $f\in\mathcal F_d$ the function $\chi_f$ defined above (restricted to monic polynomials) is a primitive Dirichlet character modulo $x^{d+1}$ of order $p$, i.e. $\chi_f^p$ is trivial. In particular $\chi_f$ is even (since $p$ is coprime with the order of $\F_q^\times$).
\item[(ii)] The map 
$$\mathcal F_d\to\{\mbox{primitive Dirichlet characters modulo }x^{d+1}\mbox{ of order }p\}$$ given by $f\mapsto \chi_f$ is a bijection.
\item[(iii)] $L(u,\chi_f)=(1-u)L(u,f,\psi)$.
\end{enumerate}
\end{proposition}

\begin{proof} This is a restatement of \cite{Ent12}*{Lemmas 7.1, 7.2}.\end{proof}

\begin{corollary}
The collection of $L$-function $\{L(u,f,\psi):f\in\mathcal F_d\}$ is independent of the choice of $\psi$. In particular if $X:\mathcal F_d\to\C$ is a function depending only on the roots of $\Lfp$ then its mean $\mean{X}_{f\in\mathcal F_d}$ is independent of the choice of the character $\psi$. The same applies to $\mathcal \AS^0_d$.
\end{corollary}

\begin{proof}
The right hand side in Proposition \ref{prop:astodir}(ii) does not depend on $\psi$. The claim for $\AS_d^0$ follows from Lemma \ref{lem:twist}.
\end{proof}

\subsection{The case of an odd polynomial A-S curve}\label{sec:odd_as_to_dir}

In the present subsection we keep the notation and setting of section \ref{sec:dir_pol}. We assume $p>2$ and $(d,2p)=1$. Note that $\AS_d^{0,odd}\subset\mathcal F_d$ ($\mathcal F_d$ is defined by (\ref{eq:fd})), in particular for $\AS_d^{0,odd}$ we have the associated character $\chi_f$ defined by (\ref{eqn_chifDef}).

\begin{lem}\label{lem:real}
For $f\in\oasfam$ we have $\Lfp\in\mathbb{R}[z]$. 
\end{lem}
\begin{proof}
See \cite{Ent12}*{Lemma 8.1}.
\end{proof}

Define \begin{equation}\label{eq:H_asodd}H=\{\chi_f|f\in\AS_e^{0,odd},1\le e\le d\}\cup\{1\}.\end{equation}

\begin{prop}\label{prop:odd_as_to_dir}
$H$ is a subgroup of the group of characters of order $p$ modulo $x^{d+1}$. Moreover, $$\{\chi_{f}|f\in\oasfam,\deg f=d\}=\prgrp{H}.$$  
\end{prop}

\begin{proof}
See \cite{Ent12}*{\S 8}.
\end{proof}

For an abelian group $A$ and a group of characters $B\subset A^*$ we denote by 
$$B^\perp=\{a\in A:\chi(a)=1\text{ for all } \chi\in B\}$$ the orthogonal group of $B$. The orthogonality relations imply then that for every $a\in A$,

\begin{equation}\label{eqn_orthogroupAvg}
    \mean{\chi(a)}_{\chi\in B}=
    \begin{cases}
    1,& a\in B^\perp,\\
    0,              & a\not\in B^\perp.
\end{cases}
\end{equation}

\begin{lem}\label{OrthoGroupLemma1}
Let $Q$ denote $x^d$ or $x^{d+1}$ and let $H_Q$ denote the set of characters $\chi\in\left((\F_q[x]/Q)^\times\right)^*$ that induce a character in $H$. Then its orthogonal group 
$H_Q^\perp=\{a\in (\F_q[x]/Q)^\times:\chi(a)=1\forall\chi\in H_Q\}$ equals
$H_Q^\perp=AB$, where $A,B\subset(\mathbb{\mathbb{F}}_{q}[x]/Q)^{\times}$
are the following subgroups:
$$
A=\{f\in\field_{q}[x]:f(0)\ne0,f(x)\equiv g(x^{2})\pmod{Q},\text{ for some }g\in\field_q[x]\},
$$
$$
B=\{f\in\mathbb{\mathbb{F}}_{q}[x]:f(0)\ne0,f(x)\equiv g(x^{p})\pmod{Q},\text{ for some }g\in\field_q[x]\}.
$$

\end{lem}
\begin{proof}
See \cite{Ent12}*{Lemma 8.5}.
\end{proof}

\subsection{Associating a Dirichlet character with an ordinary Artin-Schreier curve}\label{sec:ordas_to_dir}

Let $d$ be a natural number, $g$ a monic squarefree polynomial of degree $d$ or $d-1$ and $f=h/g\in\AS_{d}^\ord$. Recall that this condition means that $(h,g)=1$ and additionally $\deg h\le d$ if $\deg g=d$ and $\deg h=d$ if $\deg g=d-1$ (equivalently the polar divisor of $f=h/g$ has degree $d$, and no non-simple poles). 
For $f=h/g\in\AS_d^\ord$ 
and $c\in\F_q[x]$ define
\begin{equation}\label{eq:chihgc}
\chi_f(c)=\left[\begin{array}{ll}\psi\left(
\tr_{q/p}\left(\sum_{c(\alpha)=0}
f(\alpha)\right)
\right),& (c,g)=1,\\0,&(c,g)\neq 1,\end{array}\right.
\end{equation}
the summation being over roots in $\fqb$ counted with multiplicity. 
Note that for $(c,g)=1$, $\sum_{c(\alpha)=0}{f(\alpha)}$ is in $\field_{q}$, since it is a symmetric function in the roots of $c$. Moreover, $\chi_f$ is clearly a strongly multiplicative function. It is not generally a Dirichlet character, but we will see below that it is a Dirichlet character modulo $g^2$ provided $\deg h<\deg g$. The $L$-function of $\chi_f$ is defined as in (\ref{eq:def_dir_l}).

\begin{rem}Note that this definition of $\chi_f$ differs from (\ref{eqn_chifDef}). We will make it clear in the beginning of each section which of the two notions of $\chi_f$ we use, but there should be little confusion because we will only use (\ref{eqn_chifDef}) when considering the polynomial families and (\ref{eq:chihgc}) when considering the ordinary family.\end{rem}

\begin{prop} \label{prop_lhg_is_dirichlet_l}
Let $f=h/g\in\AS^\ord_{d}$. Then $L(u,\chi_f)=(1-\delta(f) u)L(u,f,\psi)$, where
\begin{equation}\label{eq:delta_f}\delta(f)=\left[\begin{array}{ll} \psi\left(\tr_{q/p}f(\infty)\right),&\deg h=\deg g=d,\\1,&\mathrm{otherwise}.\end{array}\right.\end{equation}
\end{prop}

\begin{proof}
Let $r$ be a natural number and $\alpha\in\field_{q^r}$ an element with (monic) minimal polynomial $c\in\field_q[x]$. Since $\alpha\in\field_{q^r}$ we have $\deg c|r$. Note that 
\begin{equation}\label{eq:psi_tr_chi}\psi\left(
\tr_{q^r/p}f(\alpha)
\right)=
\psiof{\frac{r}{\deg c}\tr_{q^{\deg c}/p}f(\alpha)
}=\chi_f(c)^{r/\deg c}\end{equation}
by (\ref{eq:chihgc}). 
By (\ref{eqn_LfpsiDef}) and (\ref{eq:delta_f}) we have
$$L(u,f,\psi)=\exp\left(\sum_{r=1}^\infty\left[\frac{u^r}{r}\left(\delta(f)^r+\sum_{\alpha\in\field_{q^r}, g(\alpha)\ne 0}\psi\left(\tr_{q^r/p} f(\alpha)\right)\right)\right]\right).$$
Using the multiplicativity of the function $\chi_f(c)$ and (\ref{eq:psi_tr_chi}),
\begin{multline*}
u\frac{\d\log L(u,f,\psi)}{\d u}=\sum_{r=1}^\infty u^r\left(\delta(f)^r+\sum_{\alpha\in
\field_{q^r}, g(x)\ne 0}\psi\left(\tr_{q^r/p} f(\alpha)\right)\right)
\\ =\frac {\delta(f)u}{1-\delta(f)u}+\sum_{r=1}^\infty u^r 
\sum_{c\in\Pe\atop{\deg c|r}}(\deg c)\chi_f(c)^{r/\deg c}
=\frac {\delta(f)u}{1-\delta(f)u}+\sum_{c\in\Pe}\sum_{k=1}^\infty (\deg c)\chi_f(c)^k u^{k(\deg c)}\\=\frac {\delta(f)u}{1-\delta(f)u}+\sum_{c\in\Pe}\frac{(\deg c)
\chi_f(c)u^{\deg c}}{1-\chi_f(c)u^{\deg c}}
=\frac {\delta(f)u}{1-\delta(f)u}+\sum_{c\in\Pe}u\frac{\d\log(1-\chi_f(c)u^{\deg c})^{-1}}{\d u}\\=u\frac{\d\log \left((1-\delta(f)u)^{-1}L(u,\chi_f)\right)}{\d u}.
\end{multline*}
Since $L(0,f,\psi)=L(0,\chi_f)=1$ we have $L(u,f,\psi)=(1-\delta(f)u)^{-1}L(u,\chi_f)$.
\end{proof}

It turns out that for $f=h/g\in\AS_d^\ord$ the function $\chi_f$ is a Hecke character of $\F_q(x)$ (see \cite{Ros02}*{\S 9} for this notion) with modulus $g^2$ if $\deg h\le\deg g=d$ and $g^2\cdot\infty^2$ if $\deg h>\deg g=d-1$. We will only need this fact for the case $\deg h<\deg g$, in which case $\chi_f$ is in fact a Dirichlet character modulo $g^2$. For monic squarefree $g$ of degree $d$ we denote
\begin{equation}\label{eq:def_Hg}\hg=\{f=h/g: h\in\F_q[x],(h,g)=1,\deg h<d\}\subset\AS_{d,g}^\ord.\end{equation} 

\begin{prop}\label{prop:chi_f_is_dir}
For $f=h/g\in\hg$, the function $\chi_f$ is a Dirichlet character modulo $g^2$.
\end{prop}

\begin{proof} The (strong) multiplicativity of $\chi_f$ is immediate from (\ref{eq:chihgc}), so we need to show that $\chi_{h/g}(c)$ depends only on $c$ mod $g^2$. We prove the stronger claim that 
$\sum_{c(\alpha)=0}
{h(\alpha)}/{g(\alpha)}$
depends only on $c$ (mod $g^2$). We factor $g$ over $\fqb$, writing
$
g=\prod_{i=1}^d(x-\beta_i)
$. Note that all the $\beta_i$ are distinct since we assumed $g$ is squarefree in the definition of $\AS_d^\ord$. Since $\deg h<d=\deg g$ we can write a partial fraction decomposition
$$
\frac{h(x)}{g(x)}=\sum_{i=1}^d\frac{c_i}{x-\beta_i},
$$
for some $c_i\in\overline{\field}_q$. This allows us to write
$$
\sum_{c(\alpha)=0}\frac{h(\alpha)}{g(\alpha)}=\sum_{c(\alpha)=0}\sum_{i=1}^d\frac{c_i}{\alpha-\beta_i}=\sum_{i=1}^d c_i\sum_{c(\alpha)=0}\frac{1}{\alpha-\beta_i}.
$$
We claim that for every $i$, $\sum_{c(\alpha)=0}\frac{1}{\alpha-\beta_i}$
depends only on the residue class $c\bmod{(x-\beta_i)^2}$. Then the whole sum will be determined by $c$ mod $g^2=(x-\beta_1)^2\cdots(x-\beta_d)^2$. Fix some $1\le i\le d$ and write 
$$c=a_0+a_1(x-\beta_i)+\cdots+a_r(x-\beta_i)^r,a_j\in\fqb.$$
By Vieta's formula, we have $\sum_{c(\alpha)=0}\frac 1{\alpha-\beta_i}=-a_1/a_0$. Thus it depends only on $c \bmod (x-\beta_i)^2$ and therefore the sum $\sum_{i=1}^dc_i\sum_{c(\alpha)=0}\frac 1{\alpha-\beta_i}$ depends only on $c\bmod g^2$, as required.
\end{proof}

\begin{proposition}\label{prop_distinct_characters}
The map $f\mapsto\chi_f$ gives a bijection between $\mathcal H_g$ and the set of primitive Dirichlet characters modulo $g^2$ of order $p$. In particular for $f\in\mathcal H_g$ the character $\chi_f$ is even.
\end{proposition}

\begin{proof}
First we show that $f\mapsto\chi_f$ is injective on $\mathcal H_g$.

Assume that $\chi_{f_1}=\chi_{f_2}$ for some $f_1,f_2\in\mathcal H_g$. Denote $f=f_1-f_2$ and $f=h/
g,\deg h<d=\deg g$. For every $r\ge 1$, and every $\alpha\in\field_{q^r}$ such that $g(\alpha)\ne 
0$, we have $\tr_{q^r/p}(f(\alpha))=0$ by plugging the minimal polynomial of $\alpha$ over $\F_q$ 
into ($\ref{eq:chihgc}$). By the additive version of Hilbert's Theorem 90 we have $f(\alpha)=b^p-b$ 
for some $b\in\field_{q^r}$. Fixing this $b$ we note that it is also true that $f(\alpha)=(b+i)^p-(b
+i)$ for $i\in\F_p$. Thus there are at least $p$ points defined over $\F_{q^r}$ on the curve $y^p-y=f(x)$ with 
abscissa $x=\alpha$. This is true for all $\alpha\in\field_{q^r}$ with $g(\alpha)\neq 0$, so the 
number of $\field_{q^r}$-rational points on the curve $y^p-y=f(x)$ (and hence on $C_f$) is at least $pq^r-(p-1)\deg g$. 
Note that this is true for every $r\ge 1$. This is impossible since the curve 
$C_f$ is absolutely irreducible and therefore $\#C_f(\F_{q^r})\sim q^r$ as $r\to\infty$.

Hence it remains to show that the number of order $p$ primitive characters modulo $g^2$ is exactly $\#\mathcal H_g$. It is immediate from the definition of $\mathcal H_g$ that $\#\mathcal H_g=\varphi(g)$, where $\varphi(g)=\#(\F_q[x]/g)^\times$ denotes the Euler totient function. 

To count the order $p$ primitive characters modulo $g^2$ we apply inclusion-exclusion ($\mu$ denotes the M\"obius function):
\begin{align*}
\#&\{\text{order }p\text{ primitive characters modulo }g^2\}\\
&=\sum_{Q|g^2}\mu(Q)\cdot\#\{\text{order }p\text{ characters with modulus }g^2/Q\}\\
&=\sum_{Q|g}\mu(Q)\cdot\#\{\text{order }p \text{ characters modulo }g^2/Q\}.
\end{align*}
Factor $g=Q_1\cdots Q_l$ into distinct primes. For $Q=Q_1\cdots Q_s$ we have $g^2/Q=Q_1\cdots Q_s\cdot Q_{s+1}
^2\cdots Q_l^2$ and therefore
$$(\field_q[x]/(g^2/Q))^\times\cong \bigoplus_{i=1}^l\Z/(q^{\deg Q_i}-1)\oplus\bigoplus_{i=s+1}^l(\Z/p)^{\log_pq\cdot\deg Q_i}.$$ The $p$-torsion subgroup of $(\field_q[x]/(g^2/Q))^{\times*}$ is isomorphic to the $p$-torsion subgroup of 
$(\field_q[x]/(g^2/Q))^{\times}$ which has order $\prod_{i=s+1}^lp^{\log_pq\cdot \deg Q_i}=q^{\deg(g/Q)}$.

Hence
$$
\#\{\text{order }p\text{ primitive character modulo }g^2\}=\sum_{Q|g}\mu(Q)q^{\deg (g/Q)}=\varphi(g)=\#\mathcal H_g,
$$
which completes the proof.
\end{proof}

%

\subsection{Averaging traces and trace products over primitive Dirichlet characters}
\label{sec:avtraces}

For an even primitive Dirichlet character $\chi$ modulo $Q$ with $\deg Q=m$ we recall that its $L$-function factors as follows:

\begin{equation}\label{eqn_lchiDef}
L(u,\chi)=(1-u)\prod_{i=1}^{m-2}(1-q^{1/2}\rho_iu),
\end{equation}
with $|\rho_i|=1$. For any integer $r$ define the \emph{$r$-th trace of $\chi$} by 
\begin{equation}\label{eq:def_trace}T_\chi^r=\sum_{i=1}^{m-2}\rho_i^r.\end{equation} This quantity can be interpreted as the (normalized) trace of the $r$-th power of the Frobenius acting on a certain irreducible piece of the \'etale cohomology of the Carlitz curve of $Q$, but we will not need this fact and will only use the elementary expression (\ref{eq:def_trace}). The $r$-th traces are instrumental in studying $L$-zero densities (and other $L$-zero statistics).

Applying the operator $u\frac{\d\log}{\d u}$ to \eqref{eqn_lchiDef}, using the Euler product (\ref{eq:def_dir_l}) and comparing coefficients we get for any integer $r\ge 1$:

\begin{equation}\label{eqn_expressingTracesWithCharacters}
T^r_\chi=\sum_{i=1}^{m-2}\rho_i^r=-q^{-r/2}-q^{-r/2}\sum_{P\in\mathcal{P}\atop{\deg P|r}}(\deg P)\chi(P)^{r/\deg P}=-q^{-r/2}-q^{-r/2}\sum_{\deg c=r\atop{\mathrm{monic}}}\Lambda(c)\chi(c),
\end{equation}
where 
$$\Lambda(c)=\left\{\begin{array}{ll}\deg P,&c=P^k,P\mbox{ prime},k\ge 1,\\
0,&\mbox{otherwise}.\end{array}\right.$$ is the von Mangoldt function. 
Note that $T^{-r}_\chi=\overline{T^r_\chi}$. 

Let $Q\in\field_q[x]$ be a (monic) modulus of degree $m$, and let 
$$G=\{\chi\in\left(\F_q[x]/Q\right)^{\times*}:\chi(\F_q^\times)=1\}$$ be the group of even Dirichlet characters modulo $Q$, which we view as a subgroup of the dual group of $(\field_q[x]/Q)^\times$. For a monic divisor $Q'|Q$, we denote by $G_{Q'}$ the subgroup of even Dirichlet characters with period $Q'$ (i.e. characters $\chi$ modulo $Q$ such that for $(g,Q)=1$ the value $\chi(g)$ depends only on $g\bmod Q'$). Let $H$ be a subgroup of $G$. We denote by $\prgrp{H}$ the set of primitive characters inside $H$. If $X:H\to\C$ is a function, we can express the sum $\sum_{\chi\in H^{\pr}}X(\chi)$ in terms of sums over subgroups of $H$. Denote by $H_{Q'}=H\cap G_{Q'}$ the subgroup of characters in $H$ of period $Q'$. By inclusion-exclusion we have

$$
\sum_{\chi\in H^{\pr}}X(\chi)=\sum_{Q'|Q}\mu(Q/Q')\sum_{\chi\in H_{Q'}}X(\chi),
$$
where $\mu$ is the M\"obius function, and the sum is over monic divisors $Q'$ of $Q$. The last equation allows us to express averages over primitive characters in $H$ in terms of averages over subgroups of $H$:

\begin{equation}\label{eqn_incex_principle}
\mean{X(\chi)}_{\chi\in H^{\pr}}=\sum_{Q'|Q}\mu(Q/Q')\frac{\#H_{Q'}}{\#H^{\pr}}\mean{X(\chi)}_{\chi\in H_{Q'}}.
\end{equation}

For integers $r,s$ we denote the means of traces and trace products over $\prgrp H$ by
\begin{equation}\label{eq:def_mr}M_H^r=\mean{T_{\chi}^r}_{\chi\in \prgrp{H}}, \end{equation}
\begin{equation}\label{eq:def_mrs}M_H^{r,s}=\mean{T_{\chi}^r T_{\chi}^s}_{\chi\in \prgrp{H}}.\end{equation}
For a subgroup $N\subset(\field_q[x]/Q)^\times$, we let $\psi(s;N)$ denote the following sum

\begin{equation}\label{eq:chebychev}
\psi(s;N)=\sum_{\deg c=r\atop{\mathrm{monic}\atop{c\bmod Q\in N}}}\Lambda(c)
\end{equation}
(this is the Chebychev $\psi$ function for the group $N$).

The following lemma is a slight reformulation of \cite{Ent12}*{Proposition 7.4}.

\begin{lem}\label{dirCharacterAvgLemma}
For a subgroup $H\subset (\F_q[x]/Q)^{\times*}$ and $0\neq r\in\Z$ we have
$$
M_H^r=-q^{-|r|/2}\left(1+\sum_{Q'|Q}\mu(Q/Q')\frac{\#H_{Q'}}{\#\prgrp{H}}\psi(|r|;H_{Q'}^\perp)\right),
$$
where $\sum_{Q'|Q}$ denotes summation over monic divisors $Q'$ of $Q$.
In particular $M_H^r=M_H^{-r}$ is always real.
\end{lem}

Next we want to express $M_H^{r,s}$ in a similar manner. Let $N$ be a subgroup of $(\field_q[x]/Q)^\times$. For $r,s\in\Z$ we consider the following sums (summation is over monic polynomials):

\begin{equation}\label{eq:def_tau}
\tau(r,s;N)=\left\{\begin{array}{lll}\displaystyle\sum_{\deg h=|r|,\deg g=|s|\atop hg\bmod Q\in N}\Lambda(h)\Lambda(g),& rs\ge 0,\\ \\
\displaystyle\sum_{\deg h=|r|,\deg g=|s|\atop{(gh,Q)=1\atop hg^{-1}\bmod Q\in N}}\Lambda(h)\Lambda(g),
&rs<0\end{array}\right.
\end{equation}
(note that in the case $rs\ge 0$ the summation is automatically over $h,g$ coprime with $Q$ since $hg\bmod Q\in N\subset(\F_q[x]/Q)^\times$).

\begin{lem}\label{lem_dirCharMulTrace}
For a subgroup $H\subset(\F_q[x]/Q)^{\times*}$ and $0\neq r,s\in\Z$, we have
$$
M_H^{r,s}=-q^{-\frac{|r|+|s|}{2}}-q^{-|s|/2}\cdot M_H^r-q^{-|r|/2}\cdot M_H^s+
q^{-\frac{|r|+|s|}{2}}
\sum_{Q'|Q}\mu(Q/Q')\frac{\#H_{Q'}}{\#\prgrp{H}}\tau(r,s;H_{Q'}^\perp),
$$
where the summation is over monic divisors $Q'$ of $Q$.
In particular $M_H^{r,s}=M_H^{-r,-s}$ is always real.
\end{lem}

\begin{proof}
For $r,s\ge 0$ using \eqref{eqn_expressingTracesWithCharacters} we write
\begin{align*}
q^{\frac{r+s}{2}}M_H^{r,s}&=\mean{q^{r/2}T_\chi^r\cdot q^{s/2}T_\chi^s}_{\chi\in \prgrp{H}}\\
&=\mean{
\left(
\sum_{\deg h=r}\Lambda(h)\chi(h)+1
\right)
\left(
\sum_{\deg g=s}\Lambda(g)\chi(g)+1
\right)}_{\chi\in \prgrp{H}}\\
&=-1-q^{r/2}M_H^r-q^{s/2}M_H^s+\sum_{\deg h=r,\deg g=s}\Lambda(h)\Lambda(g)\mean{\chi(hg)}_{\chi\in \prgrp{H}}.
\end{align*}

By \eqref{eqn_incex_principle} we have
$$
\mean{\chi(hg)}_{\chi\in\prgrp{H}}=\sum_{Q'|Q}\mu(Q/Q')\frac{\#H_{Q'}}{\#\prgrp{H}}\mean{\chi(hg)}_{\chi\in H_{Q'}}
$$
and by the orthogonality relation \eqref{eqn_orthogroupAvg} we have $\mean{\chi(hg)}_{\chi\in H_{Q'}}=\mathds{1}_{hg\bmod Q\in H_{Q'}^\perp}$ (for any relation $\mathcal R$ we denote by $\mathds{1}_\mathcal R$ its indicator, which takes the value 1 if $\mathcal R$ holds and 0 otherwise). Plugging this back in and exchanging the order of summation we get

$$
q^{\frac{r+s}{2}}M_H^{r,s}=-1-q^{r/2}M_H^r-q^{s/2}M_H^s+\sum_{Q'|Q}\mu(Q/Q')\frac{\#H_{Q'}}{\#\prgrp{H}}\tau(r,s;H_{Q'}^\perp).
$$
Dividing by $q^{\frac{r+s}{2}}$ we get the desired result. Note that the right hand side in the last equation is real. The case of $r\le 0,s\le 0$ follows because $M_H^{r,s}=\overline{M_H^{-r,-s}}=M_H^{-r,-s}$. 
The case $r\ge 0,s<0$ is treated similarly, only the expression $\mean{\chi(hg)}_{\chi\in H_{Q'}}$ is replaced with $\mean{\chi(h)\overline{\chi(g)}}=\mean{\chi(h/g)}$ resulting in the stated expression.
\end{proof}

\section{Counting short vectors in $\protect\field_{q}[x]$-lattices}
\label{sec:lattice}

A fundamental tool in all of our estimates of the trace means $M_H^r,M_H^{r,s}$ (defined in section \ref{sec:avtraces}) for the specific subgroups $H$ we encounter, will be bounding the number of short vectors in certain $\F_q[x]$-lattices. In the present section we summarize the background we will need, mainly the basic theory of reduced bases for $\F_q[x]$-lattices developed by Lenstra \cite{Len85}, and derive some useful bounds on the number of short vectors in a lattice. Analogous (and even more general) bounds are also known over the integers, see e.g. \cite{Kat94}*{\S 2},\cite{Sch68}*{\S 5}. 

In what follows an \emph{$\field_{q}[x]$-lattice $\Gamma$ of rank $n$} is an $\mathbb{F}_{q}[x]$-submodule
of $\mathbb{F}_{q}[x]^n$ of finite index. The volume of $\Gamma$ is defined to be the index $\vol(\Gamma)=[\field_{q}[x]^n:\Gamma]$.
 For a vector $v=(v_{1},\ldots,v_n)\in\field_{q}[x]^n$,
the degree of $v$ is defined by $\deg v=\max(\deg v_1,\ldots,\deg v_n )$. We say that $v$ is
primitive in $\Gamma$ if one cannot write $v=aw,w\in\Gamma,a\in\F_q[x]$ with $\deg a\ge 1$.

\begin{theorem}[Lenstra]\label{thm:lenstra_basis_reduction}
Let $\Gamma\subset\F_q[x]^n$ be a lattice of rank $n$ and volume $\vol(\Gamma)=q^m$. Then there exists a basis $b_1,\ldots,b_n$ for $\Gamma$ (as an $\F_q[x]$-module) such that $\deg b_1+\ldots+\deg b_n=m$. For such a basis and any $v=\sum_{i=1}^n {c_i b_i},c_i\in\F_q[x]$ we have $\deg v=\max{\deg{c_i b_i}}$. 
\end{theorem}

\begin{proof}
The first assertion is proved in \cite[Proposition 1.14]{Len85}. The second assertion is \cite[Proposition 1.1]{Len85}.
\end{proof}

The next lemma allows us to bound the number of short primitive vectors in a rank 2 lattice.

\begin{lem}\label{primitiveVectorsBoundLemma}
Let $\Gamma\subset\F_q[x]^2$ be a lattice of rank $2$, $\vol(\Gamma)=q^{m}$, $r\ge 0$ a nonnegative integer.
Then the number of primitive $v\in\Gamma$ with $\deg v \le r$ is
at most $\max(q,q^{2r-m+2})$.
\end{lem}
\begin{proof}

Let $b_1,b_2$ be a basis of $\Gamma$ as given by Theorem \ref{thm:lenstra_basis_reduction} and  denote $\deg b_{i}=m_{i}$ with $m_1+m_2=m$. Let $v=c_{1}b_{1}+c_{2}b_{2}$ be a vector with $\deg v\le r$. If
$r<m_{2}$ then since $r\ge \deg(v)=\max(\deg(c_{1}b_{1}),\deg(c_2b_2))$,
we must have $c_{2}=0$ and thus $v=c_{1}b_{1}$. If $\deg c_{1}\ge1$,
$v$ is not primitive, so there are at most $q$ possible choices of $v$. The same argument applies if $r<m_1$. If
$r\ge m_1,m_2$ then by Theorem \ref{thm:lenstra_basis_reduction} a vector $v=c_1b_1+c_2b_2\in\Gamma$ (not necessarily
primitive) of degree at most $r$ satisfies $\deg c_{i}\le r-m_{i}$. Thus the
number of vectors $v\in\Gamma$ of degree at most $r$ is bounded by $q^{2r-m_{1}-m_{2}+2}=q^{2r-m+2}$.
\end{proof}

We will also need a slight generalization:

\begin{lem}\label{primitiveVectorsBoundLemma2}
Let $\Gamma\subset\F_q[x]^2$ be a lattice of rank $2$, $\vol(\Gamma)=q^{m}$, $r,s\ge 0$ nonnegative integers.
Then the number of primitive $v=(g,h)\in\Gamma,g,h\in\F_q[x]$ with $\deg g\le r$,$\deg h\le s$ is
at most $\max(q,q^{r+s-m+2})$.
\end{lem}

\begin{proof} Assume by way of symmetry that $r\ge s$. Consider the modified lattice
$$\Gamma'=\{(g,hx^{r-s}):(g,h)\in\Gamma\}.$$ We have $\vol(\Gamma)=q^{m+r-s}$ and applying Lemma \ref{primitiveVectorsBoundLemma} to $\Gamma'$ we see that there are at most $\max(q,q^{r+s-m+2})$ primitive vectors $(g,h)$ in $\Gamma$ with $\deg((g,hx^{r-s}))\le r$, which is equivalent to the assertion.\end{proof}

\begin{lem} \label{lem:lin_eqn_lattice_vol}
\begin{enumerate}\item[(i)] Let $0\neq Q\in\F_q[x]$ be a polynomial and let $a\in\F_q[x]$ be a polynomial coprime with $Q$. Then 
$$\Gamma=\{(g,h)\in\F_q[x]^2:h\equiv ag\pmod Q\}$$
is a lattice of volume $q^{\deg Q}$. 
\item[(ii)] Let $r,s\ge 0$ be integers such that $r+s\ge\deg Q-1$. Then there exists a solution $(g,h)$ to the congruence $h\equiv ag\pmod Q$ with $(h,g)\not\equiv(0,0)\pmod Q$ and $\deg g\le r,\deg h\le s$.
\end{enumerate}
\end{lem}

\begin{proof}
{\bf (i)} It is easy to see that $\Gamma$ is spanned by the vectors $(1, a),(0,Q)$. It is shown in \cite[\S1]{Len85} that the volume the lattice in $\F_q[x]^n$ spanned by the rows of a matrix $A\in\mathrm{M}_{n\times n}(\F_q[x])$ is given by $q^{\deg(\det A)}$ (provided $\det A\neq 0$), which equals $q^{\deg Q}$ in our case. 

{\bf (ii)} Denote $V=\{(g,h)\in\F_q[x],\deg g\le r,\deg h\le s\}$. This is an $\F_q$-vector space of dimension $r+s+2$. The map $V\to\F_q[x]/Q:(g,h)\mapsto h-ag\bmod Q$ cannot be injective (since its range has smaller dimension than its domain) and so has a nonzero vector in its kernel, giving the desired solution.
\end{proof}

The following bound will be used in section \ref{sec:ord_as_fam}.

\begin{lem}\label{lem:bound_short_vectors} Let $\Gamma\subset\F_q[x]^n$ be a lattice with volume $\vol(\Gamma)=q^m$ and denote $\mu=\min_{0\neq v\in\Gamma}\deg v$. Then for any real $s\ge 0$,
$$\#\{v\in\Gamma:\deg v\le s\}\le q^{\max(0,n(s+1)-m,(n-1)(s+1-\mu))}.$$
\end{lem}

\begin{proof} Let $b_1,\ldots,b_n$ be a basis for $\Gamma$ as in Theorem \ref{thm:lenstra_basis_reduction}, so each $v\in\Gamma$ can be written uniquely as $v=\sum_{i=1}^nc_ib_i$ with $\deg v=\max\deg c_ib_i$. First assume that $\max\deg b_i\le s$. Then for $v$ as above with $\deg v\le s$ we must have $\deg c_i\le s-\deg b_i$ and we have $q^{s-\deg b_i+1}$ possible values for $c_i$ and $q^{n(s+1)-\sum b_i}=q^{n(s+1)-m}$ possible values of $v$ in total.

On the other hand if $\deg b_i>s$ for some $i$, say $\deg b_n>s$, then for $v$ as above with $\deg v\le s$ we must have $c_n=0$ and $\deg c_i\le s-\deg b_i\le s-\mu$ for $1\le i\le n-1$. It follows that there are $\max(1,q^{(n-1)(s+1-\mu)})$ possible values of $v$.
\end{proof}

\section{2-level density for the polynomial A-S family}\label{sec_full_as_family}

In the present section we prove Theorem \ref{thm:main1}.
Throughout the present section $p>2$ is prime, $q$ is a power of $p$, $\psi:\F_p\to\C^\times$ a nontrivial additive character. Let $d$ be a natural number with $(d,p)=1$. Until section \ref{sec:2levelfinish} the asymptotic $O(\cdot)$ notation will always have an absolute implicit constant (independent of $q,d$). 

Let $L(u,f,\psi)$ be an A-S $L$-function (not necessarily polynomial) with factorization 
$$L(u,f,\psi)=\prod_{j=1}^{2\g/(p-1)}(1-q^{-1/2}\rho_ju),\g=\g(C_f).$$ For $r\in\Z$ define the \emph{$r$-th trace} of $L(u,f,\psi)$ to be
\begin{equation}\label{eq:trace_of_as} T_f^r=\sum_{j=1}^{2\g/(p-1)}\rho_j^r\end{equation}
($\psi$ is implicit in the notation).
Let $\mathcal F_d\subset\AS_d^0$ be the subfamily defined by (\ref{eq:fd}). We will prove Theorem \ref{thm:main1} through estimating the trace-product averages $$M_d^{r,s}=\mean{T_f^r T_f^s}_{\mathcal F_d}.$$

An estimate of $M_d^{r,s}$ was given in \cite{Ent12}*{Proposition 6.3}, but only in the range $|r|+|s|<d$. In the present section we give an estimate which is nontrivial in the wider range $|r|+|s|<(2-2/p)d$. This will allow us to obtain the improved range in Theorem \ref{thm:main1}.
 
By Proposition \ref{prop:astodir} we have $M_d^{r,s}=M_H^{r,s}$, where 
$$H=\left\{\chi\in\left(\F_q[x]/x^{d+1}\right)^{\times*}:\chi^p=1\right\}$$
and $M_H^{r,s}$ is given by (\ref{eq:def_mrs}).
By Lemma \ref{lem_dirCharMulTrace} and the fact that $\#H_{x^{d+1}}/\#\prgrp{H}=\frac{q}{q-1}$, $\#H_{x^d}/\#\prgrp{H}=\frac{1}{q-1}$ (here we use $(d,p)=1$) we have
\begin{equation}\label{eq:mrsviatau}
 M_H^{r,s}=-q^{-\frac{|r|+|s|}{2}}-q^{-|s|/2}M_H^r-q^{-|r|/2}M_H^s+q^{-\frac{|r|+|s|}{2}}\bigg[
\frac{q}{q-1}\tau(r,s;H_{x^{d+1}}^\perp)-
\frac{1}{q-1}\tau(r,s;H_{x^{d}}^\perp)
\bigg]. 
\end{equation}

For the rest of this section we let $Q$ denote $x^d$ or $x^{d+1}$. We need to describe the orthogonal group to $H_{Q}$. 
\begin{lem}\label{secondOrthoGrpRepLemma}
A polynomial $u(x)\in\field_{q}[x]$, $u(0)\ne0$ satisfies $u\bmod Q\in H_{Q}^{\perp}$
if and only if $u\equiv\varphi(x^{p})\pmod{Q}$ for some $\varphi\in\F_q[x]$ with $\varphi(0)\neq 0,\deg\varphi\le d/p$. 
\end{lem}

\begin{proof} The group $H_Q$ consists of all degree $p$ characters of $\left(\F_q[x]/Q\right)^\times$ and therefore $H_Q^\perp=\left(\left(\F_q[x]/Q\right)^\times\right)^p$. Since $Q$ is $x^d$ or $x^{d+1}$ and the $p$-th powers in $\F_q[x]$ are exactly $\F_q[x^p]$. Noting that $(d,p)=1$ the assertion follows.\end{proof}

We will need the estimate of $$M_d^r=\langle T_f^r\rangle_{f\in\mathcal F_d}=M_H^r$$ given in \cite{Ent12}*{Theorem 1}:

\begin{prop}\label{prop_AlexeiResult}
$M_d^r=O\left(q^{(1/p-1/2)r}+rq^{r/2-(1-1/p)d}\right)$ (for $r>0$).
\end{prop}

As a corollary we obtain for $r,s\neq 0$ (by (\ref{eq:mrsviatau}) and the fact that $M_H^r=M_H^{-r}$ which follows from Lemma \ref{dirCharacterAvgLemma}):
\begin{multline}\label{eq:MHrs}
M_H^{r,s}=O\left(q^{\frac{\max(|r|,|s|)}{p}-\frac{|r|+|s|}{2}}+\max(|r|,|s|)q^{\left||r|-|s|\right|/2-d(1-1/p)}\right)+\\ +q^{-\frac{|r|+|s|}{2}}\bigg[
\frac{q}{q-1}\tau(r,s;H_{x^{d+1}}^\perp)-
\frac{1}{q-1}\tau(r,s;H_{x^{d}}^\perp)
\bigg].
\end{multline}
Thus, to estimate $M_H^{r,s}$ we must estimate the sums $\tau(r,s;H_Q^\perp)$. There are two fundamentally different cases: the \emph{diagonal case} $r=-s$ and the \emph{off-diagonal case} $r\neq -s$.

\subsection{Estimating the diagonal terms $M_d^{r,-r}$}

In the present subsection we prove the following estimate for $r>0$:
\begin{equation}\label{prop_diagonalTermBound}
M_d^{r,-r}=M_H^{r,-r}=r+O(r^2q^{r-(1-1/p)d+3}+rq^{-r/2}).
\end{equation}
Note that for any fixed $\epsilon>0$ and $\epsilon d<r<(1-1/p-\epsilon)d$ the error term is small compared with the main term.

Recall that
$$
\tau(r,-r;H_Q^\perp)=\sum_{\deg g=\deg h=r\atop{(gh,x)=1\atop{gh^{-1}\bmod Q\in H_Q^\perp}}}\Lambda(g)\Lambda(h),
$$
the sum being over monic $g,h$.
We separate the pairs $(g,h)\in\field_q[x]^2$ for which the residue $gh^{-1}\pmod{Q}$ is in $H_Q^\perp$ into two types: first there are pairs for which $g=h$ (diagonal pairs), and then there are pairs with $g\ne h$ (off-diagonal pairs). The contribution of diagonal pairs is easy to estimate using the Prime Polynomial Theorem (see below). 

The key observation for bounding the contribution of the off-diagonal pairs $g\neq h$ is that such pairs $(g,h)$ are primitive vectors of degree $r$ in a lattice of the form 
\begin{equation}\label{eq:gammaphi}\Gamma_\varphi=\{(a,b)\in\field_q[x]^2:b(x)\equiv \varphi(x^p)a(x)\pmod{Q}\},\end{equation} with $\varphi\in\F_q[x],\deg\varphi\le d/p$.
Indeed, since we have $hg^{-1}\bmod{Q}\in H_Q^\perp$, by Lemma \ref{secondOrthoGrpRepLemma} we have a congruence $hg^{-1}\equiv \varphi(x^p)\pmod{Q}$ for some $\varphi$ coprime with $x$ with $\deg\varphi\le d/p$. Since $g,h$ are prime powers of the same degree, $g\neq h$ implies $(g,h)=1$, i.e. the vector $(g,h)$ is primitive in $\F_q[x]^2$ and \emph{a fortiori} in $\Gamma_\varphi$.

Denoting the contribution of off-diagonal pairs by
   
$$E(r)=\sum_{\deg g=\deg h=r\atop{(gh,x)=1\atop{ gh^{-1}\bmod Q\in  H_Q^\perp}}}\Lambda(g)\Lambda(h)-\sum_{\deg g=r\atop{g\neq x^r}}\Lambda(g)^2.$$

We see that
$$
E(r)=O(r^2)\cdot\#\{\text{Primitive vectors } (u,v)\text{ in }\Gamma_\varphi\text{ of degree }r\text{, for some } \varphi\in\field_q[x], \deg \varphi\le d/p\}.
$$

Fix $\varphi$ with $(\varphi,x)=1,\deg\varphi\le d/p$. By Lemma \ref{lem:lin_eqn_lattice_vol}, the volume of the lattice $\Gamma_\varphi$ is $q^{\deg Q}$, which is $q^d$ or $q^{d+1}$. By Lemma \ref{primitiveVectorsBoundLemma} the number of primitive vectors in $\Gamma_\varphi$ of degree $r$ is at most $O(q^{2r-d+2})$. Since there are $O(q^{d/p+1})$ possible values of $\varphi$, we have $E(r)=O(r^2 q^{2r-d(1-1/p)+3})$. Thus,

\begin{equation}
\tau(r,-r;H_Q^\perp)=\sum_{\deg g=r}\Lambda(g)^2+O(r^2 q^{2r-d(1-1/p)+3}).
\end{equation}
By the Prime Polynomial Theorem (see e.g. \cite{Ros02}*{Theorem 2.2}) we have 
$$
\sum_{\deg g=r}\Lambda(g)^2=rq^r+O(rq^{r/2}),
$$
and therefore
$$
\tau(r,-r; H_Q^\perp)=r q^r+O(r q^{r/2}+r^2 q^{2r-(1-1/p)d+3}).
$$
Recalling (\ref{eq:MHrs}) we finally obtain
\begin{align*}
M_H^{r,-r}&=O(q^{-(1-1/p)r}+rq^{-(1-1/p)d})
+\\
&+q^{-r}\bigg[\frac{q}{q-1} rq^r-\frac{1}{q-1} rq^r+O(r q^{r/2}+r^2 q^{2r-(1-1/p)d+3})\bigg]\\
&=r+O(r q^{-r/2}+r^2 q^{r-(1-1/p)d+3}),
\end{align*}
which is exactly (\ref{prop_diagonalTermBound}).

\subsection{Bounding the off-diagonal terms $M_d^{r,s},r\neq -s$}

In the present subsection we prove the following estimates, valid for $r,s\neq 0,r\neq -s$:
\begin{equation}\label{eq:offdiag1}M_{d}^{r,s}=M_H^{r,s}= O\left(|rs| q^{(1/p-1/2)(|r|+|s|)}+|rs|q^{\frac{|r|+|s|}{2}-d(1-1/p)+3}\right).\end{equation}
Since $M_d^{r,s}=M_d^{s,r}=M_d^{-r,-s}$ we may assume without loss of generality that $r>0,|r|\ge|s|$. 

\maincase{The case of $M_d^{r,-s}$, $r>s>0, r\ne s$} Lemma \ref{lem_dirCharMulTrace} gives:
\begin{equation}\label{eq:offdiag_Mtau}
M_H^{r,-s}=-q^{-\frac{r+s}{2}}-q^{-s/2} M_H^r-q^{-r/2}M_H^s+q^{-\frac{r+s}{2}}\bigg[
\frac{q}{q-1}\tau(r,-s;H_{x^{d+1}}^\perp)-
\frac{1}{q-1}\tau(r,-s;H_{x^d}^\perp)
\bigg].
\end{equation}
By Proposition \ref{prop_AlexeiResult} we have
\begin{equation}\label{eq:offdiag_linear_est}
M_H^r=M_d^r=O(q^{(1/p-1/2)r}+r q^{r/2-d(1-1/p)}),
\end{equation}
so it remains to bound the contribution of the sums $\tau(r,-s;H_Q^\perp)$ for $Q=x^d,x^{d+1}$.
By (\ref{eq:def_tau}) and since $Q$ is a power of $x$ we have
$$
\tau(r,-s;H_Q^\perp)=\sum_{\deg h=r,\deg g=s\atop{(gh,x)=1\atop hg^{-1}\bmod Q\in H_Q^\perp}}\Lambda(h)\Lambda(g).
$$
We treat coprime pairs with $\gcd(h,g)=1$ and non-coprime pairs $\gcd(h,g)\ne 1$ separately.

\case{Coprime pairs $\gcd(g,h)=1$} Let $g,h$ be a coprime pair as above, that is $\deg g=r,\deg h=s,gh^{-1}\bmod Q\in H_Q^\perp,\gcd(g,h)=1$. By Lemma \ref{secondOrthoGrpRepLemma}, for every such pair there is a polynomial $\varphi$ with $(\varphi,x)=1,\deg\varphi\le d/p$ such that
\begin{equation}\label{hg_lattice_eqn}
h(x)\equiv g(x)\varphi(x^p)\pmod{Q},
\end{equation}
and therefore $(g,h)$ is a primitive vector in the lattice $\Gamma_\varphi$ (defined in (\ref{eq:gammaphi})) with $\deg g=r,\deg h=s$.
There are at most $q^{d/p+1}$ possible values of $\varphi$ and by Lemma \ref{primitiveVectorsBoundLemma2} and the fact that $\mathrm{vol}(\Gamma_\varphi)=q^{\deg Q}\ge q^d$, for a given $\varphi$ there are at most $O(q^{r+s-d+2})$ possibilities for $(g,h)$.
In total there are $O\left(q^{s+r-d+d/p+3}\right)$ possibilities for the triple $\varphi,g,h$ and 
hence also for the number of coprime pairs $g,h$ as above. The total contribution of such pairs to 
$\tau(r,-s,H_Q^\perp)$ is therefore
$$O\left(rsq^{{s+r}-(1-1/p)d+3}\right).$$

\case{Non-coprime pairs $\gcd(g,h)\neq 1$} It is left to bound the contribution of the non-coprime pairs $(g,h)$ to $\tau(r,-s;H_Q^\perp)$. Such pairs are of the form $g,h$ which are a power of a common prime $x\neq P\in\P$. Write $g=P^{r/\deg P}, h=P^{s/\deg P}$. The condition $gh^{-1} \bmod{Q}\in H_Q^\perp$ is equivalent to
$
P^{(r-s)/\deg P}\bmod{Q}\in H_Q^\perp
$. Since $H_Q^\perp=\left((\F_q[x]/Q)^\times\right)^{p}$ this condition is automatic if $p|(r-s)/\deg P$ and is equivalent to $P\bmod Q$ being a $p$-th power if $p\nmid (r-s)/\deg P$

If $p|(r-s)/\deg P$ then $\deg P\le (r-s)/p$ and there are $O(q^\frac{r-s}p)$ suitable $P$, contributing a total of $O(s^2q^\frac{r-s}p)$ to the sum $\tau(r,-s,H_Q^\perp)$.
On the other hand if $p\nmid (r-s)/\deg P$ we must have $P\equiv\varphi(x)^p\pmod{Q}$, for some $
\varphi\in\field_q[x]$ of degree $\le d/p$ (since $P$ is prime this is only possible if $\deg P\ge d
$). There are $O(q^{d/p+1})$ possible values of $\varphi$ and having $\varphi$ fixed, $P$ has at most 
$q^{s-d}$ possible values (note that $\deg P\le s$ since $h$ is a power of $P$). Overall, the number of 
possible $P$ is at most $O(q^{s-(1-1/p)d+1})$ and their contribution to $\tau(s,-r,H_Q^\perp)$ is $O(s^2q^{s-(1-1/p)d+1})$.

We conclude that
$$
\tau(r,-s;H_Q^\perp)=O\left( rsq^{s+r-d(1-1/p)+3}+s^2q^{(r-s)/p}\right).
$$
Overall we get, using (\ref{eq:offdiag_Mtau}) and (\ref{eq:offdiag_linear_est}),
\begin{align*}
M_{H}^{r,-s} = O\left(q^{\frac rp-\frac{r+s}2}+rsq^{\frac{s+r}2 - (1-\frac 1p)d+3}+
s^2q^{\frac{r-s}p-\frac{r+s}2}\right),
\end{align*}
confirming (\ref{eq:offdiag1}) in this case.

\maincase{The case $M_{d}^{r,s}$, $r\ge s>0$} Recall that
$$\tau(r,s;H_Q^\perp)=\sum_{\deg g=r,\deg h=s\atop{gh\bmod Q\in H_Q^\perp}}\Lambda(g)\Lambda(h).$$ Once again we separate the contribution of coprime pairs (that is, pairs with $\gcd(g,h)=1$) and non-coprime pairs. 

\case{Coprime pairs $\gcd(g,h)=1$} We start with the case of coprime pairs. By Lemma \ref{secondOrthoGrpRepLemma} such pairs satisfy $gh\equiv\varphi(x^p) \pmod{Q}$, where $\varphi$ is a polynomial of degree $<\deg Q/p\le d/p$ such that $\varphi(0)\ne 0$. 

There are at most $q^{d/p+1}$ possible values of $\varphi$ and each $\varphi$ determines $q^{r+s-\deg Q}\le q^{r+s-d}$ possibilities for the product $w=gh$ (note that if $r+s<\deg Q$ then since $g,h$ are coprime prime powers and $\varphi(x^p)$ is a $p$-th power there are no suitable pairs). Each $w$ has at most two decompositions into coprime prime powers and we are left with $O(q^{r+s-d+d/p+1})$ pairs $(g,h)$ contributing $O(rsq^{r+s-(1-1/p)d+1})$ to $\tau(r,s;H_Q^\perp)$.

\case{Non-coprime pairs $\gcd(g,h)\neq 1$} If $(g,h)$ is a non-coprime pair then $g,h$ are powers of a common prime $P$ and are uniquely determined by the product $w=gh$. If $r+s<\deg Q$ then $w=\varphi(x^p)$ with $\varphi$ monic and there are $O(q^{(r+s)/p})$ possibilities for $w$ contributing $O(s^2q^{(r+s)/p+1})$ to $\tau(r,s;H_Q^\perp)$. On the other hand if $r+s\ge\deg Q$ then there are $O(q^{d/p+1})$ possible $\varphi$ and $O(q^{r+s-d})$ possible $w$ for each $\varphi$, hence there are $O(q^{r+s-(1-1/p)d+1})$ possible $w$, so we get an overall contribution of $O(s^2q^{r+s-(1-1/p)d+1})$ to $\tau(r,s;H_Q^\perp)$.

Overall we obtain
$$\tau(r,s;H_Q^\perp)=O\left(rsq^{r+s-(1-1/p)d+1}+s^2q^{(r+s)/p}\right),
$$
and therefore using (\ref{eq:MHrs}) and (\ref{eq:offdiag_linear_est}) that
$$M^{r,s}_d=O\left( rsq^{\frac{r+s}2-(1-1/p)d+1}+s^2q^{\left(\frac 1p-\frac 12\right)(r+s)}\right),$$
confirming (\ref{eq:offdiag1}) in this case as well. We have verified (\ref{eq:offdiag1}) in all cases.

\subsection{2-level density for the polynomial A-S family}\label{sec:2levelfinish}

The results of the last section allow us to prove Theorem \ref{thm:main1} by a standard calculation involving Fourier series. Let $\Phi\in\Scal(\R^2)$ be a fixed test function. In the present section the implicit constants in the $O$-notation may depend on the test function $\Phi$, i.e $O(\cdot)=O_\Phi(\cdot)$. All asymptotic notation is in the limit $d\to\infty$ and is uniform in $q$.

\begin{proof}[Proof of Theorem \ref{thm:main1}]

Denote $\mathcal M_d^{r,s}=\mean{T_f^rT_f^s}_{f\in\AS_d^0}.$ and recall the notation $M_d^{r,s}=\mean{T_f^rT_f^s}_{f\in\mathcal F_d}.$ By Lemma \ref{lem:twist} combined with (\ref{eq:as_fd}) and (\ref{eq:trace_of_as}) we have
\begin{equation}\label{eq:rel_masfd2}\mathcal M_d^{r,s}=\mean{\psi(\tr_{q/p}b)^{r+s}}_{b\in\F_q}\mean{T_f^rT_f^s}_{f\in\mathcal F_d}=\mathds 1_{p|r+s}\cdot M_d^{r,s}\end{equation} ($\mathds 1_{p|r+s}=1$ if $p|r+s$ and 0 otherwise), hence the same estimates as in (\ref{prop_diagonalTermBound}) and (\ref{eq:offdiag1}) apply to $\mathcal M_d^{r,s}$. Similarly if we denote $\mathcal M_d^r=\mean{T_f^r}_{f\in\AS_d^0}$ then 
\begin{equation}\label{eq:rel_masfd1}\mathcal M_d^r=\mathds 1_{p|r}\cdot M_d^r\end{equation} and the estimate in Proposition \ref{prop_AlexeiResult} applies to $\mathcal M_d^r$ as well.

Now if $\hat{\Phi}$ is supported on $|\eta|+|\xi|<2-2/p$
then in fact, $\hat{\Phi}$ is supported on $|\eta|+|\xi|<2-2/p-2\delta$
for a fixed $\delta>0$.
By Fourier expansion we have (see \cite[Corollary 6.5]{Ent12} for the full calculation)
\[
\mean{\Wtwofp}_{f\in\mathcal F_d}=\frac{1}{(d-1)^{2}}\sum_{r,s=-\infty}^{\infty}\hat{\Phi}
\left(\frac{r}{d-1},\frac{s}{d-1}\right)(M_{d}^{r,s}-M_{d}^{r+s}).
\]
and the same calculation shows
$$\mean{\Wtwofp}_{f\in\AS_d^0}=\frac{1}{(d-1)^{2}}\sum_{r,s=-\infty}^{\infty}\hat{\Phi}
\left(\frac{r}{d-1},\frac{s}{d-1}\right)(\mathcal M_{d}^{r,s}-\mathcal M_{d}^{r+s}).$$
Since $\hat{\Phi}$ is supported on $|\eta|+|\xi|<2-2/p-2\delta$,
\[
\mean{\Wtwofp}_{f\in\asfam}=\frac{1}{(d-1)^{2}}\sum_{|r|+|s|<2d(1-1/p-\delta)}\hat{\Phi}\left(\frac{r}{d-1},\frac{s}{d-1}\right)(\mathcal M_{d}^{r,s}-\mathcal M_{d}^{r+s}).
\]

First we bound the contribution to the sum of $r,s$ with $r\ne-s$.
Note that by (\ref{eq:rel_masfd1}) and Proposition \ref{prop_AlexeiResult} we have, taking $k=|r+s|$,
\begin{align*}
\sum_{|r|+|s|<2d(1-1/p-\delta)\atop r\ne -s}|\mathcal M_{d}^{r+s}|=&
\sum_{|r|+|s|<2d(1-1/p-\delta)\atop r\ne -s}O\left(q^{\frac{|r+s|}{p}-\frac{|r+s|}{2}}+|r+s|q^{\frac{|r+s|}{2}-d(1-1/p)}\right)\\
=&\sum_{k=1}^{2d(1-1/p-\delta)}O\left(dq^{k/p-k/2}+dkq^{k/2-d(1-1/p)}\right)=O(d+d^2q^{-\delta d})=O(d).
\end{align*}

Similarly, by (\ref{eq:rel_masfd2}) and \eqref{eq:offdiag1},
\begin{align*}
\sum_{|r|+|s|<2d(1-1/p-\delta)\atop{r\neq -s}}|\mathcal M_d^{r,s}|=&
\sum_{|r|+|s|<2d(1-1/p-\delta)}O\left(|rs|q^{\frac{|r|+|s|}{p}-\frac{|r|+|s|}{2}}+|rs|q^{\frac{|r|+|s|}{2}-d(1-1/p)+3}\right)\\
=&\sum_{k=1}^{2d(1-1/p-\delta)}O(dk^2 q^{k/p-k/2})+O(d^3q^{-\delta d+3})=O(d).
\end{align*}

Hence as $\hat{\Phi}$ is bounded (continuous function with compact support) we
get that the overall contribution to $\mean{\Wtwofp}_{f\in\asfam}$
of the terms with $r\ne-s$ is $O(1/d)$. It remains to estimate the contribution of the $r=-s$ terms, which using (\ref{eq:rel_masfd1}) and \eqref{eq:offdiag1} is
$$\frac{1}{(d-1)^{2}}\sum_{-d(1-1/p-\delta)<r<d(1-1/p-\delta)}\hat{\Phi}\left(\frac{r}{d-1},\frac{-r}{d-1}\right)(\mathcal M_{d}^{r,-r}-\mathcal M_{d}^{0})
$$
$$=\frac{d-2}{d-1}\hat{\Phi}(0,0)+\frac{1}{(d-1)^2}\sum_{0\ne|r|<d(1-1/p-\delta)} \hat{\Phi}\left(\frac{r}{d-1},\frac{-r}{d-1}\right)(\mathcal M_{d}^{r,-r}-d+1),
$$
(note that $\mathcal M_{d}^{0}=d-1,\mathcal M_{d}^{0,0}=(d-1)^{2}$). By (\ref{prop_diagonalTermBound}) and (\ref{eq:rel_masfd2}) we have $$\mathcal M_{d}^{r,-r}=r+O(r^2q^{r-d(1-1/p)+3}+rq^{-r/2}).$$
As above we can bound the overall contribution of the error term by $O(1/d)$. Combining the estimates we get
\begin{multline*}
\mean{\Wtwofp}_{f\in\asfam}\\
=\frac{d-2}{d-1}\hat{\Phi}(0,0)+\frac{1}{(d-1)^{2}}\sum_{|r|<d(1-1/p-\delta)}\hat{\Phi}\left(\frac{r}{d-1},\frac{-r}{d-1}\right)(|r|-d+1)+O(1/d)\\
=\frac{d-2}{d-1}\hat{\Phi}(0,0)+\sum_{|r|<d(1-1/p-\delta)}\hat{\Phi}\left(\frac{r}{d-1},\frac{-r}{d-1}\right)\left(\frac{|r|}{d-1}-1\right)\frac{1}{d-1}+O(1/d)\\
\to_{d\to\infty}\hat{\Phi}(0,0)+\int_{-\infty}^{\infty}\Phi(\sigma,-\sigma)(|\sigma|-1)d\sigma=\hat{\Phi}(0,0)-\int_{-\infty}^{\infty}\hat{\Phi}(\sigma,-\sigma)\max(1-|\sigma|,0)d\sigma,
\end{multline*}
by the definition of the Riemann integral (we use
the fact that $\hat{\Phi}(\sigma,-\sigma)$ is supported on $|\sigma|<1-1/p-\delta$).
By \cite{Ent12}*{Lemma 6.6}, this is equal to the desired limit 
$$
\int_{-\infty}^{\infty}\int_{-\infty}^{\infty}\Phi(t,s)\left(1-\left(\frac{\sin\left(\pi(t-s)\right)}{\pi(t-s)}\right)^{2}\right)\d t\d s.
$$
\end{proof}

\section{The odd polynomial A-S family}\label{sec_odd_poly}

In the present section we study the 1-level density of the odd polynomial A-S family $\oasfam$ and prove Theorem \ref{thm:main2}. Once again the theorem will follow by a standard Fourier series calculation from a good estimate of $\tmean{{T_{f}^r}}_{f\in\oasfam}$ in a suitable range, where $\chi_f$ is defined by (\ref{eq:trace_of_as}) and $T_{f}^r$ is defined by (\ref{eq:def_trace}). We obtain a good estimate for $\tmean{{T_{f}^r}}_{f\in\oasfam}$ in the range $r<d(1-1/p)$, significantly strengthening the result of \cite[Theorem 4]{Ent12} where the range of $r$ is only logarithmic in $d$. This allows us to compute the 1-level density for this family when the test function has Fourier transform supported on $\left(-(1-1/p),1-1/p\right)$.

Throughout the present section we assume $(d,2p)=1$ and denote
$M_{d}^{r}=\langle{T_{f}^{r}}\rangle_{f\in\oasfam}$. We also adopt the notation of sections \ref{sec:dir_pol}, \ref{sec:odd_as_to_dir} and \ref{sec:avtraces}. In particular for $f\in\AS_d^{0,\odd}$ the associated Dirichlet character $\chi_f$ is defined by (\ref{eqn_chifDef}). Until section \ref{sec:odd_density} the implicit constant in the asymptotic $O(\cdot)$ notation is absolute.

\subsection{Estimating $M_d^r$}
In the present subsection we will prove the following
\begin{prop}\label{prop_tracesMeanAsymptotic} For $r\ge 1$,
$$M_{d}^{r}=\mathds{1}_{2|r}+O\left(rq^{-(r/6)+3\omega}+rq^{\frac 12(r-(1-1/p)d)+3}\right),$$
where $\mathds{1}_{2|r}$ is 1 if $r$ is even and 0 if $r$ is odd and $\omega=\mathds 1_{r>d/4-1}$ is 1 if $r>d/4-1$ and 0 otherwise.
\end{prop}

Note that this estimate is only useful if $r<(1-1/p)d$. We begin with the observation that by Proposition \ref{prop:astodir} combined with (\ref{eq:trace_of_as}) and (\ref{eq:def_trace}) we have $T^r_f=T^r_{\chi_f}$ for $f\in\AS_d^{0,\odd}$.
Let $H$ be the group of Dirichlet characters defined by (\ref{eq:H_asodd}). One easily computes from Lemma \ref{OrthoGroupLemma1} and Proposition \ref{prop:odd_as_to_dir} that
$$\#H=\#H_{x^{d+1}} = q^{\frac{d+1}2-\lfloor\frac dp\rfloor+\lfloor\frac d{2p}\rfloor},\quad\#H_{x^d}=q^{\frac{d-1}2-\lfloor\frac dp\rfloor+\lfloor\frac d{2p}\rfloor},\quad
\#H^{\mathrm{pr}}=(q-1)q^{\frac{d-1}2-\lfloor\frac dp\rfloor+\lfloor\frac d{2p}\rfloor}.$$
Lemma \ref{dirCharacterAvgLemma} now implies
\begin{equation}\label{AppliedLemma2}
M_d^r=M_H^r=\mean{T_\chi^r}_{\chi\in \prgrp{H}}=-q^{-r/2}\left(1+
\frac{q}{q-1}\psi(r;H_{x^{d+1}}^\perp)-\frac{1}{q-1}\psi(r;H_{x^d}^\perp)
\right).
\end{equation}
For the rest of the present section let $Q=x^d$ or $x^{d+1}$. By (\ref{AppliedLemma2}), Proposition \ref{prop_tracesMeanAsymptotic} would follow from the following estimate:

\begin{prop}\label{prop_etaBound}
Let $r\in\N$. We have $$\psi(r; H_Q^\perp)=\mathds{1}_{2|r}\cdot q^{r/2}+O\left(rq^{(r/3)+3\omega}+rq^{r-(1/2-1/2p)d+3}\right),$$
where $\omega=\mathds 1_{r>d/4-1}$.
\end{prop}

The proof of Proposition \ref{prop_etaBound} (from which Proposition \ref{prop_tracesMeanAsymptotic} follows) occupies the rest of the present subsection.
Using the notation of Lemma \ref{OrthoGroupLemma1} we let $A,B\subset(\mathbb{\mathbb{F}}_{q}[x]/Q)^{\times}$
be the following subgroups:
$$
A=\{f\in(\field_{q}[x]/Q)^\times:f(0)\ne0,f(x)\equiv g(x^{2})\pmod{Q},\text{ for some }g\in\field_q[x]\},
$$
$$
B=\{f\in(\mathbb{\mathbb{F}}_{q}[x]/Q)^\times:f(0)\ne0,f(x)\equiv g(x^{p})\pmod{Q},\text{ for some }g\in\field_q[x]\}.
$$
By Proposition \ref{OrthoGroupLemma1}, $H_Q^{\perp}=AB$. Our strategy in bounding the size of $\psi(r; H_Q^\perp)$ is as follows: we will estimate the "diagonal" contribution to $\psi(r; H_Q^\perp)$ of polynomials of the form $f(x^2)$ using the Chebotarev Density Theorem (in an elementary case which actually follows from the Prime Polynomial Theorem) and bound the "off-diagonal" contribution coming from $f\bmod Q\in AB\setminus A$. Note that the diagonal contribution is exactly $\psi(r;A)$.

\begin{prop}[Off-diagonal estimate]\label{prop_nonChebotarevBound}
We have (recall that $d$ is odd)
$$\psi(r; H_Q^\perp)-\psi(r;A)=O\left(rq^{r-(1-1/p)\frac{d}{2}+3}+rq^{\frac rp+3\omega}\right),$$
where $\omega=\mathds 1_{r>d/4-1}$.
\end{prop} 

%

\begin{proof}
Note that (for $r<d$)
$$\psi(r; H_Q^\perp)-\psi(r;A)=\sum_{\deg c=r\atop{\mathrm{monic}\atop{c\bmod Q\in AB\setminus A}}}\Lambda(c).$$
We will assume that $d\ge p$, otherwise the sum is empty. Let $c\in\F_q[x]$ be a (monic) prime power such that $\deg c=r<d$ and $c\bmod Q\in H_Q^\perp=AB\setminus A$. By the definition of $A,B$ one can write
\begin{equation}\label{eq:fg1g2}
c(x)\equiv g_1(x^2)g_2(x^p)\equiv \left(\sum_{i=0}^l a_i x^{2i}\right)\left(\sum_{j=0}^s b_j x^{pj}\right)\pmod{Q}.
\end{equation}

The minimal odd power that can appear in $c$ is $x^p$, and so writing $c$ as the sum of an odd polynomial and an even polynomial we will have \begin{equation}\label{eq:fu1u2}c(x)=u_{1}(x^{2})+x^{p}u_{2}(x^{2}),\,u_2\neq 0\end{equation} (here $u_2\neq 0$ because we are assuming $c\bmod Q\not\in A$ and $\deg c=r<d\le\deg Q$). Plugging this into (\ref{eq:fg1g2}) we obtain
$
u_1(x^2)+x^pu_2(x^2)\equiv g_1(x^2)g_2(x^p)\pmod{Q}
$. Divide this equation by $g_1(x^2)$ and note that the inverse of $g_1(x^2)$ modulo $Q$ has the form $h(x^2)$ for some $h\in\F_q[x]$ (since $A$ is a group). We get
\begin{equation}\label{hu1u2_eqn}
h(x^2)u_1(x^2)+x^ph(x^2)u_2(x^2)\equiv g_2(x^p) \pmod{Q}.
\end{equation}
Separating into even and odd exponents we can write
$$
h(x^2)u_1(x^2)\equiv h_1(x^{2p}),\, x^ph(x^2)u_2(x^2)\equiv x^ph_2(x^{2p})\pmod{Q},
$$ for some $h_1,h_2\in\F_q[x]$.
Equivalently (recall that $Q$ is $x^d$ or $x^{d+1}$, $(d,2p)=1$ and $p$ is odd), \begin{equation}\label{eq:def_h1}h(x)u_1(x)\equiv h_1(x^p) \pmod{x^{(d+1)/2}},\end{equation}
\begin{equation}\label{eq:def_h2}h(x)u_2(x)\equiv h_2(x^{p})\pmod{x^{(d-p)/2}}.\end{equation}

Since $x\nmid c$ we must have  $x\nmid g_2$, and by (\ref{eq:fg1g2}), (\ref{hu1u2_eqn}) and (\ref{eq:def_h1}) also $x\nmid h,u_1,h_1$. Thus we can divide (\ref{eq:def_h2}) by (\ref{eq:def_h1}) to obtain a congruence
\begin{equation}\label{eq:phicong}
\frac{u_2(x)}{u_1(x)}\equiv \frac{h_2(x^p)}{h_1(x^p)}\equiv\varphi(x^p)\pmod{x^{(d-p)/2}},
\end{equation}
for some $\varphi\in\F_q[x]$ with $\deg\varphi<(d-p)/2p$. 

Denote by $\Gamma_\varphi\subset\F_q[x]^2$ the lattice of solutions $(U_1,U_2)$ to the congruence
\begin{equation}\label{eq:congruence}
U_2\equiv U_1\varphi(x^p)\pmod{x^{(d-p)/2}}.
\end{equation}
By Lemma \ref{lem:lin_eqn_lattice_vol} we have $\mathrm{vol}(\Gamma_\varphi)=q^{(d-p)/2}$. 
Now $(u_1, u_2)$ is a vector on $\Gamma_\varphi$ with $\deg u_1\le\floor{\frac r2},\deg u_2\le\floor{\frac{r-p}2}$ (by (\ref{eq:fu1u2})). It is 
also a primitive vector, since if $u_1,u_2$ had a nonconstant common divisor by (\ref{eq:fu1u2}) the polynomial $c$ 
could not be a prime power (indeed if $w$ is such a common divisor then $w(x^2)|c$ and if $c$ is a 
power of a prime $P\neq x$ then $w(x^2)$ is also a power of $P$ and we must have $P\in\F_q[x^2]$ and $c\in
\F_q[x^2]$, a contradiction to (\ref{eq:fu1u2}) and $u_2\neq 0$). By Lemma \ref{primitiveVectorsBoundLemma}, for given $\varphi$ there are 
at most $$q^{\max(1,\floor{r/2}+\floor{(r-p)/2}-(d-p)/2+2)}$$ suitable primitive solutions $(u_1,u_2)$ of (\ref{eq:congruence}). We have 
$O(q^\floor{d/2p+1/2})$ possibilities for $\varphi$ (since $\deg\varphi<(d-p)/2p$), so overall the number of possible $(u_1,u_2)$ is 
\begin{equation}\label{eq:full_basis_bound} O\left(q^{\max\left(1,\frac r2+\floor{\frac{r-p}2}-\frac{d-p}2+2\right)+\floor{\frac d{2p}+\frac 12}}\right).\end{equation} Since $c$ is completely determined by $u_1,u_2$, this is also a bound on the number 
of possible $c$. We now subdivide our analysis into three ranges for $r$ with respect to $d,p$ (the ranges we consider are exhaustive because $d$ is odd).

\case{The case $r\ge d/2$} In this case (\ref{eq:full_basis_bound}) is $O(q^{r-(1-1/p)\frac{d}
{2}+3})$, which implies the assertion of the proposition since each $c$ contributes $\Lambda(c)\le r
$ to the sum.

\case{The case $d/2-p-1< r< d/2$} If $r<p$ then (\ref{eq:fg1g2}) cannot be satisfied, hence we may assume $r\ge p$. Since also $r>d/2-p-1$, we have $r>d/4-1$ and so $\omega=1$ in the notation of Proposition \ref{prop_nonChebotarevBound}. We also have
$\frac d{2p}+\frac 12<\frac rp+2$, 
hence (\ref{eq:full_basis_bound}) is $O\left(q^{r/p+3}\right)$ and noting that $\Lambda(c)\le r$ we obtain the stated bound.

\case{The case $r<d/2-p-1$} For $\varphi$ as above we note that by Lemma \ref{lem:lin_eqn_lattice_vol}(ii),
for any given $l,m$ with $l+m\ge\floor{(d-p)/2p}$ we may find a nonzero vector $(F_1,F_2)\in\field_q[x]^2$ 
with $\deg F_1\le l,\deg F_2\le m$ satisfying $$F_2\equiv\varphi F_1\pmod{x^{\floor{(d-p)/2p}+1}}.$$ Then we 
have $F_2(x^p)\equiv \varphi(x^p)F_1(x^p)\pmod {x^{(d-p)/2}}$ and \eqref{eq:phicong} implies \begin{equation}\label{eq:congruence_uF}u_1(x) 
F_2(x^p)\equiv u_2(x) F_1(x^p)\pmod{x^{(d-p)/2}}.\end{equation} Note that since $\deg u_1\le\floor{\frac r2},\deg u_2\le\floor{\frac{r-p}2}$, we have 
\begin{equation}\label{eq:deg_inequality}\deg(u_1(x)F_2(x^p))\le\floor{\frac r2}+mp,\quad
\deg(u_2(x)F_1(x^p))\le\left\lfloor{\frac{r-p}2}\right\rfloor+lp.\end{equation}
Now set $$l=\floor{\frac{\frac{d-p}2-1-\floor{\frac {r-p}2}}{p}},\quad m=\floor{\frac{\frac{d-p}2-1-\floor{\frac {r}2}}{p}}.$$ A simple calculation (using the fact that $d,p$ are odd and the assumption $r<d/2-p-1$) shows that $l+m\ge\floor{(d-p)/2p}$ and therefore there exist $F_1,F_2$ (not both zero) as above.

With our choice of $l,m$ and using (\ref{eq:deg_inequality}), both sides of the congruence (\ref{eq:congruence_uF}) have degrees strictly smaller than $(d-p)/2$, and we get that in fact an equality holds. Due to the primitivity of the vector $(u_1,u_2)$ we get that in fact $u_1=aF_2(x^p)$ and $u_2=aF_1(x^p)$ for some constant $0\neq a\in\F_q$. This implies (by (\ref{eq:fg1g2})) that $c$ is a $p$-th power, hence the number of possible $c$ is at most $q^{r/p}$, which implies the assertion since $\Lambda(c)\le r$.
\end{proof}

\begin{prop}[Diagonal estimate]\label{prop_ChebotarevApplication}
Let $r<d(1-1/p)$. Then $$\psi(r;A)=r\cdot\#\{c(x)\in\field_{q}[x]:c(x^2)\mbox{ irreducible},\deg c=r/2\}+O(q^{r/4})=\mathds{1}_{2|r}\cdot q^{r/2}+O(q^{r/4}).$$
\end{prop}

\begin{proof}

The set of irreducible polynomials
$c\in\field_{q}[x]$ for which $c(x^{2})$ is irreducible in $\mathbb{F}_{q}[x]$
is precisely the set of inert primes in the extension $\field_{q}(x)/\field_{q}(x^{2})$ (we identify  prime ideals with their monic generators).

Note that $\mathrm{Gal}(\field_{q}(x)/\field_{q}(x^{2}))\cong \Z/2\Z$
and that the Frobenius element of a prime in this extension is nontrivial iff the
prime is inert. Chebotarev's density theorem \cite{Ros02}*{Theorem 9.13B} now implies\footnote{Here we could also just use the Prime Polynomial Theorem since (for $c$ irreducible) $c(x^2)$ is reducible iff $c(x^2)=h(x)h(-x)$ for an irreducible $h\in\F_q[x]$}
\begin{align*}
\#\{c & \in\field_{q}[x^2]\text{ is irreducible},\deg c=r\}=\#\{c(x^{2})\in\field_{q}[x^{2}]:c\text{ is inert,}\deg c=r/2\}\\
&=\frac{1}{2}\cdot\frac{q^{r/2}}{r/2}+O(q^{r/4}/r)=\frac{q^{r/2}}{r}+O(q^{r/4}/r),
\end{align*}
and the proposition follows (the contribution of prime powers to $\psi(r;A)$ is also $O(q^{r/4})$).
\end{proof}

\begin{proof}[Proof of Proposition \ref{prop_etaBound}]
Combine Proposition \ref{prop_nonChebotarevBound} and Proposition \ref{prop_ChebotarevApplication} (and recall that $p\ge 3$).
\end{proof}

\begin{proof}[Proof of Proposition \ref{prop_tracesMeanAsymptotic}] Combine (\ref{AppliedLemma2}) with Proposition \ref{prop_etaBound}.

\end{proof}

\subsection{1-level density for the odd polynomial family}
\label{sec:odd_density}
%

Let $\Phi\in\mathcal{S}(\R)$ be a fixed (Schwartz) test function with $\mathrm{supp}(\Phi)\subset(-1+1/p,1-1/p)$. For the remainder of the present section, the constants implicit in the asymptotic $O$-notation may depend on the test function $\Phi$, but no other parameter. In other words $O(\cdot)=O_\Phi(\cdot)$. We keep the notation $M_d^r=\tmean{T_f^r}_{f\in\AS_d^{0,\odd}}$ used throughout the present section. Recall that for $f\in\AS_d^{0,\odd}$ we have $\g(C_f)=(p-1)(d-1)/2$ and $L(u,f,\psi)$ has $d-1$ zeros. We denote by $\phi_{d-1}$ the periodic sampling function associated with $\Phi$, given by (\ref{eq:periodic}).

\begin{proof}[Proof of Theorem \ref{thm:main2}]
First we show that
\begin{equation}\label{eqn_FourierExpnansion}
\mean{\Wonefp}_{f\in\oasfam}=\hat{\Phi}(0)+\sum_{r=1}^{\infty}\left(\hat{\Phi}\left(\frac{r}{d-1}\right)+\hat{\Phi}\left(-\frac{r}{d-1}\right)\right)\frac{M_{d}^{r}}{d-1}.    
\end{equation}
Let $f\in\AS_d^{0,\odd}$ and $\rho_j=e^{2\pi i\theta_j}$ the normalized inverse zeros of $L(u,f,\psi)$. By Fourier expansion we have
$$\phi_{d-1}(\theta_{j})=\sum_{r=-\infty}^{\infty}\hat{\phi}_{d-1}(r)e^{2\pi ir\theta_j}=\sum_{r=-\infty}^{\infty}\frac 1{d-1}\hat{\Phi}\left(\frac r{d-1}\right)e^{2\pi ir\theta_j},$$ 
since $\hat{\phi}_{d-1}(r)=\frac{1}{d-1}\hat{\Phi}\left(\frac{r}{d-1}\right)$.
Hence $$\Wonefp=\sum_{r=-\infty}^{\infty}\frac{1}{d-1}\hat{\Phi}\left(\frac{r}{d-1}\right)T_{f}^{r}.$$
Averaging over $\AS_d^{0,\odd}$ and noting that (by Lemma \ref{lem:real}) $M_{d}^{r}=\overline{M_{d}^{r}}=M_{d}^{-r}$
we get \eqref{eqn_FourierExpnansion}. Let $\delta>0$ be such that $\mathrm{supp}(\Phi)\subset(-1+1/p+\delta,1-1/p-\delta)$. Then by (\ref{eqn_FourierExpnansion}) and Proposition \ref{prop_tracesMeanAsymptotic} we have for $d\ge d_0(\delta)$ large enough
\begin{multline*}
\mean{\Wonefp}_{f\in\oasfam}  =\hat{\Phi}(0)+\sum_{r=1}^{(1-1/p-\delta)d}\left(
\hat{\Phi}\left(
-\frac{r}{d-1}\right)
+\hat{\Phi}\left(
\frac{r}{d-1}\right)
\right)\frac{M_{d}^{r}}{d-1}\\
=\hat{\Phi}(0)+\frac{1}{d-1}\cdot\sum_{r=1}^{(1-1/p-\delta)d}\left(
\hat{\Phi}\left(
-\frac{r}{d-1}\right)
+\hat{\Phi}\left(
\frac{r}{d-1}
\right)\right)\left(-\mathds{1}_{2|r}+O\left(rq^{-r/6+3\cdot \mathds 1_{r>d/4-1}}+rq^{-\delta d/2+3}\right)\right)\\
=\hat{\Phi}(0)-\sum_{r=1}^{(1-1/p-\delta)d}\frac{\mathds{1}_{2|r}}{d-1}\left(
\hat{\Phi}\left(
-\frac{r}{d-1}
\right)+\hat{\Phi}\left(
\frac{r}{d-1}\right)
\right)+O\left(\frac 1d\right).
\end{multline*}

Note that on the other hand for $1\le r\le d-1$,
$$\int_{U\in\USp(d-1)}\tr (U^r)\mathrm d U=-\mathds{1}_{2|r}$$
(see \cite{DiSh94}*{Theorem 6}) and a similar calculation to (\ref{eqn_FourierExpnansion}) shows
\begin{multline*}
\mean{W_1(U,\Phi)}_{U\in \mathrm{USp}(d-1)}=\sum_{r=-\infty}^{\infty}\frac{1}{d-1}\hat{\Phi}\left(
\frac{r}{d}\right)\int_{\mathrm{USp}(d-1)}\tr(U^{r})\mathrm dU\\=\hat{\Phi}(0)-\sum_{r=1}^{(1-1/p-\delta)d}\frac{\mathds{1}_{2|r}}{d-1}\left(
\hat{\Phi}\left(
-\frac{r}{d-1}
\right)+\hat{\Phi}\left(
\frac{r}{d-1}\right)
\right)+O\left(\frac 1d\right)=\mean{\Wonefp}_{f\in\oasfam}+O\left(\frac 1d\right).$$
\end{multline*}
Taking $d\to\infty$ and using (\ref{1levelUSp}) we obtain Theorem \ref{thm:main2}.
\end{proof}

\section{The ordinary A-S family}\label{sec:ord_as_fam}

In the present section we prove Theorem \ref{thm:main3}. Throughout the section we use the notation from sections \ref{sec:ordas_to_dir} and \ref{sec:avtraces}.

Let $d$ be a natural number, $g\in\F_q[x]$ a monic squarefree polynomial of degree $d$ or $d-1$. For each of the families $\mathcal F=\AS^\ord_{d,g},\H_g$ ($\mathcal H_g$ is defined by (\ref{eq:def_Hg})) and $r\in\Z$ we denote the trace average
$$M^r_{\mathcal F}=\mean {T^r_{f}}_{f\in\mathcal F},$$ where $T^r_f=\sum_{i=1}^{2d-2}\rho_i(f,\psi)^r$, $\rho_i(f,\psi)$ being the normalized inverse zeros of $L(u,f,\psi)$. Our strategy for estimating $M^r_{\AS^\ord_{d,g}}$ will be by reducing the problem to estimating $M^r_{\mathcal H_g}$ (this will only work directly for $\deg g=d$ and slight modifications to our arguments will be necessary for the case $\deg g=d-1$). Theorem \ref{thm:main3} will be derived from these estimates.

\subsection{Estimating means of traces for the family $\hg$}
\label{sec:hg}

Assume $g$ is squarefree with $\deg g=d$ and denote by 
$$H=\left\{\chi\in(F_q[x]/g^2)^{\times*}:\chi^p=1\right\}$$ the group of order $p$ Dirichlet characters modulo $g^2$.
By Proposition \ref{prop_lhg_is_dirichlet_l} combined with Proposition \ref{prop_distinct_characters}  we have (using the notation (\ref{eq:def_mr})) $M^r_{\mathcal H_g}=M^r_H$.

We recall from section \ref{sec:avtraces} that for a group of Dirichlet characters $H$ modulo $g^2$, and $Q|g^2$ a monic divisor, we denote by $H_{Q}\subset H$ the subgroup of Dirichlet characters in $H$ with period $Q$.

\begin{prop}\label{prop_inclusionExclusionMgr}
Let $r$ be a positive integer. Then for $r>0$,
$$M^r_{\mathcal H_g}=M_{H}^r=-q^{-r/2}\left(
1+\sum_{Q|g}\mu(g/Q)\frac{q^{\deg Q}}{\varphi(g)}\psi\left(r; (\F_q[x]/gQ)^{\times p}\right)
\right),$$
where $\varphi(g)=\#(\F_q[x]/g)^\times$ is the Euler totient function and $\psi(r;*)$ is defined by (\ref{eq:chebychev}).
\end{prop}

\begin{proof} Since $g$ is squarefree, by Lemma \ref{dirCharacterAvgLemma} applied to $H$ we have (the terms with $\mu(g^2/Q')\neq 0$ have the form $Q'=gQ$ with $Q|g$)
\begin{multline*}
M_{H}^r=-q^{-r/2}
\left(
1+\sum_{Q'|g^2}\mu(g^2/Q')\frac{\#H_{Q'}}{\#\prgrp{H}}\psi(r; H_{Q'}^\perp)
\right)=-q^{-r/2}\left(
1+\sum_{Q|g}\mu(g/Q)\frac{\#H_{gQ}}{\#\prgrp{H}}\psi\left(r; H_{gQ}^\perp\right)
\right),
\end{multline*}
Now the proposition follows from the elementary facts (see proof of Proposition \ref{prop_distinct_characters}) $\#H^\pr=\varphi(g),\#H_{gQ}=q^{\deg Q}$ and
$$H_{gQ}^\perp=\left\{F\bmod g^2: F\bmod gQ\in(\F_q[x]/gQ)^{\times p}\right\}$$
(note that $(F,g^2)=1$ iff $(F,gQ)=1$).
\end{proof}

Hence we turn our attention to bounding $\psi\left(r;(\F_q[x]/gQ)^{\times p}\right)$ for $Q|g$. The key observation is that if $c\in\F_q[x]$ is a $p$-th power modulo $gQ$ then writing $c=u^p+UQg$ and taking derivatives we get that $c'=U'Qg+UQ'g+UQg'$ is divisible by $Q$ (since $Q|g$). This will allow us to give a good bound on the number of suitable $c$ of given degree. Our method crucially relies on this special property of the group $(\F_q[x]/g^2)^{\times p}$ and we would not be able to obtain a comparable result for a general group of Dirichlet characters of similar size in terms of their (common) modulus. The next lemma contains the key estimate:

\begin{lem}\label{claim_bounding_number_of_derivatives}
Let $Q$ be a squarefree polynomial. We have
$$\#\{c\in\field_q[x]: \deg c\le r,Q|c'\}\le 
\left[\begin{array}{ll}q^{\frac r2+\left(\frac 12-\frac 1p\right)(r-2\deg Q)+2p},& \deg Q<r<2\deg Q,\\ q^{\frac r2+\frac 12(r-2\deg Q)+2p},&r\ge 2\deg Q.\end{array}\right.$$
\end{lem}

The proof of the lemma (to be given below) is based on viewing the set $\{c':c\in\F_q[x]\}\cap Q\F_q[x]$ as a lattice over the ring $\F_q[x^p]$ and bounding the number of short vectors in this lattice. Denote $R=\field_q[x^p]$. Consider the (free) $R$-module
$$
V=\{c':c\in\F_q[x]\}=R\oplus xR \oplus x^2 R \oplus \ldots \oplus x^{p-2} R\subset\F_q[x].
$$
For a vector $v=\sum_{i=0}^{p-2}a_i(x^p)x^i\in V$ we define its degree to be $\deg_R{v}=\max(\deg a_i)$. We will denote by $\deg v=\deg_x v$ the degree of $v$ as a polynomial in $\F_q[x]$. Note that we always have the inequalities \begin{equation}\label{eq:deg_inequality1} \frac{\deg_xv}p-1<\deg_R(v)\le\frac{\deg_xv}p.\end{equation} For a squarefree polynomial $Q$ consider the $R$-lattice $\Lambda_Q=Q\field_q[x]\cap V$.

\begin{lem}
Assume $Q$ is a squarefree polynomial. As an $R$-lattice in $V$, $\Lambda_Q$ has volume $\vol(\Lambda_Q)=q^{\deg{Q}}$.
\end{lem}

\begin{proof}
We want to calculate $[V:\Lambda_Q]$. By the second isomorphism theorem for $R$-modules,
$$
V/\Lambda_Q\cong (V+Q\field_q[x])/Q\field_q[x].
$$
We show that $V+Q\field_q[x]=\field_q[x]$ and it would follow that $[V:\Lambda_Q]=q^{\deg{Q}}$. 

Write 
\begin{equation}\label{eq:Q_identity}Q=a_0(x^p)+xa_1(x^p)+\ldots+x^{p-1}a_{p-1}(x^p),a_i\in\F_q[x].\end{equation}
Fix $0\le i\le p-1$. Multiplying (\ref{eq:Q_identity}) by $x^{p-1-i}$ and subtracting polynomials from $V$ we get that the polynomial $a_i(x^p)x^{p-1}$ is in $V+Q\field_q[x]$. Now we note that $\gcd(a_0,\ldots,a_{p-1})=1$, since otherwise $Q$ is not squarefree. Consequently there are $c_0,\ldots,c_{p-1}\in\F_q[x]$ such that $\sum_{i=0}^{p-1} c_i a_i=1$. Hence
$$x^{p-1}=\sum_{i=0}^{p-1}c_i(x^p)a_i(x^p)x^{p-1}\in V+Q\F_q[x]$$
(since the latter is an $\F_q[x^p]$-module).
Now since $1,x,\ldots,x^{p-2}\in V$, we have that $1,x,\ldots,x^{p-1}\in V+Q\F_q[x]$. But this is a basis of $\F_q[x]$ as an $R$-module and we have $V+Q\F_q[x]=\F_q[x]$ as needed.
\end{proof}

\begin{lem}\label{prop_shortbasis_count_lambdaQ} For $Q$ squarefree,
$$\#\{v\in\Lambda_Q:\deg_x v\le r-1\}\le \left[\begin{array}{ll}q^{\left(1-\frac 2p\right)r-\left(1-\frac 2p\right)(\deg Q)+2p-1},&\deg Q<r<2\deg Q,\\ q^{\left(1-\frac 1p\right)r-\deg Q+2p-1},&r\ge2\deg Q.\end{array}\right.$$
\end{lem}

\begin{proof}
First note that for any $0\neq v\in\Lambda_Q$ we have $\deg_R(v)>\frac{\deg v}{p}-1\ge\frac{\deg Q}{p}-1$ (from (\ref{eq:deg_inequality1})). Now 
applying Lemma \ref{lem:bound_short_vectors} to $\Lambda_Q$ (with $n=p-1$, $s=(r-1)/p$, $\mu=
\min_{0\neq v\in\Lambda_Q}\deg_R v>\frac{\deg Q}p-1$) we obtain (once again using (\ref{eq:deg_inequality1}))
\begin{multline*}\#\{v\in\Lambda_Q:\deg_x v\le r-1\}\le\#\{v\in\Lambda_Q:\deg_Rv\le (r-1)/p\}\le\\
\le q^{\max\left(0,(p-1)(\frac rp+1)-\deg Q,(p-2)(\frac rp+2-\frac{\deg Q}p)\right)}\le q^{\max\left(0,\left(1-\frac 1p\right)r-\deg Q+2p-1,\left(1-\frac 2p\right)r-\left(1-\frac 2p\right)(\deg Q)+2p-1\right)},\end{multline*}
which implies the assertion.
\end{proof}

\begin{proof}[Proof of Lemma \ref{claim_bounding_number_of_derivatives}]
We note that the conditions $\deg c\le r,Q|c'$ are equivalent to $c'\in\Lambda_Q$. Since the kernel of the map $c\mapsto c'$ consists of the polynomials lying in $R=\F_q[x^p]$, any fixed value of $c'$ with $\deg c'\le r-1$ corresponds to at most $q^{r/p+1}$ values of $c$. Combining this observation with Lemma \ref{prop_shortbasis_count_lambdaQ} gives Lemma \ref{claim_bounding_number_of_derivatives}.\end{proof}

\begin{lem}\label{lem_psiHgBound}
Let $g$ be squarefree, $Q|g$. Then 
\begin{equation}\label{eq:psiHgBound}\psi\left(r;(\F_q[x]/gQ)^{\times p}\right)\le \left[\begin{array}{ll}q^{r/p},&r\le\deg Q,
\\rq^{\frac r2+\left(\frac 12-\frac 1p\right)(r-2\deg Q)+\min(2p,r/p)},&\deg Q<r<2\deg Q,\\
rq^{\frac r2+\frac 12(r-2\deg Q)+\min(2p,r/p)},&r\ge 2\deg Q.\end{array}\right.\end{equation}

\end{lem}

\begin{proof}
Since $Q|g$, we have $\psi\left(r;(\F_q[x]/gQ)^{\times p}\right)\le \psi\left(r;(\F_q[x]/Q^2)^{\times p}\right)$, so it is enough to bound the latter by the right hand side of (\ref{eq:psiHgBound}). Let $c\in\field_q[x]$ be a prime power with $\deg c = r$ such that $c\bmod Q^2$ is a $p$-th power, i.e. $c\equiv u^p\pmod{Q^2}$. Writing explicitly $c-u^p=UQ^2$ for some $U\in\field_q[x]$ and taking derivatives, we get 
$c'=U'Q^2+2UQQ'=Q(QU'+2Q'U)$, hence $Q|c'$. 

First we assume $r>\deg Q$. If $2p<r/p$ we need to obtain the bounds with $2p$ in place of the minima in the exponents. Applying Lemma \ref{claim_bounding_number_of_derivatives} and using $\Lambda(c)\le r$ we obtain the required bound. If $2p>r/p$ then we can replace Lemma \ref{claim_bounding_number_of_derivatives} with the easier bound
$$\#\{c\in\F_q[x]:c\mbox{ monic},\deg c=r,Q|c'\}\le q^{r-\deg Q+r/p}$$ which follows by noting that there are at most $q^{r-\deg Q}$ possible values of $c'$ (since $\deg c'\le r-1,Q|c'$ and we are assuming $\deg Q>r$) and for each of them there are at most $q^{r/p+1}$ possible values of $c$. The requirement that $c$ is monic reduces the total exponent by 1. From this point we argue as in the case $2p<r/p$.

Finally if $r\le\deg Q$, since $Q|c'$ and $\deg c'<r=\deg Q$ we must have $c'=0$ and $c\in\F_q[x^p]$ can be written as $c=b^p$ with $\Lambda(c)=\Lambda(b)$. In this case we have (by the Prime Polynomial Theorem) $$\psi\left(r;(\F_q[x]/Q^2)^{\times p}\right) \le\sum_{\deg b=r/p}\Lambda(b)=\mathds 1_{p|r}\cdot\psi(r/p)\le q^{r/p}$$ as required.

\end{proof}

We are now ready to derive our main bound on $\mhg^r$, which applies in the range $r\le 2d$.

\begin{prop}\label{prop:mhg}
Let $g$ be a squarefree polynomial of degree $d$ and $1\le r\le 2d$. Then
$$
\mhg^r=O\left(\frac {q^{2d}}{\varphi(g)^2}q^{(1/p-1/2)r}+\frac {r\tau(g)q^{d}}{\varphi(g)}q^{(1-1/p)(r-2d)+\min(2p,r/p)}\right),
$$
where $\varphi(g)=\#(\F_q[x]/g)^\times$ is the Euler totient function and $\tau(g)$ is the number of monic divisors of $g$.
\end{prop}

\begin{rem} We will see below that $\frac{q^d}{\varphi(g)}=O_\epsilon(d^\epsilon)$ and $\tau(g)=O_\epsilon(q^{\epsilon d})$ for any $\epsilon>0$, so these factors will not affect our calculation of (the limit of) 1-level density.\end{rem}

\begin{proof}
By Proposition \ref{prop_inclusionExclusionMgr} we have
$$
|\mhg^r|\le q^{-r/2}\sum_{Q|g}\frac{q^{\deg Q}}{\varphi(g)}\psi\left(r; (\F_q[x]/gQ)^{\times p}\right)+O(q^{-r/2}).
$$
We can now separately bound the contribution from the different ranges appearing in Lemma 
\ref{lem_psiHgBound}. We factor $g=P_1\ldots P_l$, where $P_i$ are distinct primes. 

The contribution of 
the range $r\le\deg Q$ is bounded by (Lemma \ref{lem_psiHgBound})
\begin{multline*}
\frac {q^d}{\varphi(g)}q^{(1/p-1/2)r}
\sum_{Q|g}q^{-\deg(g/Q)}=\frac {q^d}{\phi(g)}
q^{(1/p-1/2)r}\sum_{Q|g}q^{-\deg Q}=\frac {q^d}{\varphi(g)}q^{(1/p-1/2)r}\prod_{i=1}^l\left(1+q^{-\deg 
P_i}\right)
 \\ \le \frac {q^{2d}}{\varphi(g)^2}q^{(1/p-1/2)r},\end{multline*}
since $$\varphi(g)\prod_{i=1}^l\left(1+q^{-\deg P_i}\right)=q^d\prod_{i=1}^l\left(1-q^{-\deg P_i}\right)\left(1+q^{\deg P_i}\right)\le q^d.$$
The contribution of the range $\deg Q<r<2\deg Q$ is bounded by (Lemma \ref{lem_psiHgBound} and $\deg Q\le d$)
\begin{multline*}\frac {q^d}{\varphi(g)}\sum_{Q|g}q^{\deg Q-d}rq^{(1/2-1/p)(r-2\deg Q)+\min(2p,r/p)}
 \\ \le \frac {rq^d}{\varphi(g)}\sum_{Q|g}q^{(1/2-1/p)(r-2d)+\min(2p,r/p)}
\le \frac {r\tau(g)q^d}{\varphi(g)}q^{(1/2-1/p)(r-2d)+\min(2p,r/p)}.\end{multline*}
Similarly, the contribution of $r\ge 2\deg Q$ is bounded by (Lemma \ref{lem_psiHgBound} and recall that $r\le 2d$)
$$\frac {r\tau(g)q^d}{\varphi(g)}q^{\frac 12(r-2d)+\min(2p,r/p)}\le\frac {r\tau(g)q^d}{\varphi(g)}q^{(1/2-1/p)(r-2d)+\min(2p,r/p)}.$$
Combining the above three bounds we obtain the stated bound.
\end{proof}

\subsection{Means of traces for the family $\AS^\ord_{d,g},\deg g=d$}

Throughout the present subsection $d$ is a natural number and $g\in\F_q[x]$ is a monic squarefree polynomial with $\deg g=d$. Recall that
$$\AS_{d,g}^\ord=\{f=h/g:h\in\F_q[x],\deg h=d,(h,g)=1\},$$
$$\mathcal H_g=\{f=h/g:h\in\F_q[x],\deg h<d,(h,g)=1\}.$$
The family $\AS_{d,g}^\ord$ decomposes as follows: \begin{equation}\label{eq:asdg_decomposition}\AS_{d,g}^\ord=\bigsqcup_{b\in\F_q}\{f+b:f\in\mathcal H_g\}\end{equation} (disjoint union).

\begin{prop}\label{prop:masdg}
Let $g$ be a squarefree polynomial of degree $d$ and $1\le r< 2d$.
$$
M^r_{\AS_{d,g}^\ord}=O\left(\frac {q^{2d}}{\varphi(g)^2}q^{(1/p-1/2)r}+\frac {r\tau(g)q^{d}}{\varphi(g)}q^{(1-1/p)(r-2d)+\min(2p,r/p)}\right),
$$
with $\varphi,\tau$ as in Proposition \ref{prop:mhg}.
\end{prop}

\begin{proof} By (\ref{eq:asdg_decomposition}) and Lemma \ref{lem:twist},
$$M^r_{\AS_{d,g}^\ord}=\frac 1{\#\AS_{d,g}^\ord}\sum_{f\in \AS_{d,g}^\ord}T^r_f=
\frac 1{q\cdot\#\mathcal H_g}\sum_{f_1\in \mathcal H_g,b\in\F_q}\psi(b)^rT^r_{f_1} =
\langle \psi(\tr_{q/p}b)^r\rangle_{b\in\F_q}\cdot M^r_{\mathcal H_g}=\mathds 1_{p|r}\cdot M^r_{\mathcal H_g},$$
and the assertion follows from Proposition \ref{prop:mhg}.\end{proof}

\subsection{Means of traces for the family $\AS^\ord_{d,g},\deg g=d-1$}

For the present subsection assume that $g\in\F_q[x]$ is monic, squarefree with $\deg g=d-1$. In this case we cannot directly derive a bound on $M^r_{\AS_{d,g}^\ord}$ from our bound on $\mathcal H_g$. In the present subsection we explain how the arguments in section \ref{sec:hg} need to be modified to produce a similar bound on $M^r_{\AS_{d,g}^\ord}$ in the present case.

\begin{rem}If there exists $\alpha\in\F_q$ with $g(\alpha)\neq 0$ we could apply a M\"obius transformation (with coefficients in $\F_q$) to the variable $x$ that moves $\alpha$ to $\infty$ and reduce the problem back to the case $\deg g=d$. However it can happen that no such $\alpha$ exists.\end{rem}

Recall that in section \ref{sec:ordas_to_dir} we defined the functions $\chi_f:\F_q[x]\to\C$ for $f=\F_q(x)$ by (\ref{eq:chihgc}). It is evident from the definition that $\chi_{f_1+f_2}=\chi_{f_1}\chi_{f_2}$. For $a\in\F_q$ we introduce the following definitions.
$$\Lambda^a(c)=\chi_{ax}(c)\Lambda(c),\: c\in\F_q[x]\mbox{ monic}$$ ($\Lambda$ is the von Mangoldt function),
$$\psi^a(r;N)=\sum_{\deg c=r}\Lambda^a(c),$$ where $N\subset(\F_q[x]/Q)^\times$ for some modulus $Q$. Define also
$$\mathcal{H}_g^a=\{f+ax:f\in\mathcal H_g\}.$$
Note that $\mathcal{H}_g^0=\mathcal H_g$ and (recall $\deg g=d-1$)
\begin{equation}\label{eq:union_hga} \AS_{d,g}^\ord=\bigsqcup_{0\neq a\in\F_q\atop{b\in\F_q}}\{f+b:f\in\mathcal H_g^a\}\end{equation}
(disjoint union).

By slightly modifying the proof of Proposition \ref{prop_inclusionExclusionMgr} (by way of Proposition \ref{prop_lhg_is_dirichlet_l} applied to $\chi_{f+ax}=\chi_{ax}\chi_f$ for $f\in\mathcal H_g$ and noting that $\delta(f+ax)=1$ in the notation of (\ref{eq:delta_f})) one can establish the following generalization:
$$M^r_{\mathcal H_g^a}=-q^{-r/2}\left(
1+\sum_{Q|g}\mu(g/Q)\frac{q^{\deg Q}}{\varphi(g)}\psi^a\left(r; (\F_q[x]/gQ)^{\times p}\right)
\right).$$
From here one can argue essentially verbatim as in section \ref{sec:hg} replacing $\mathcal H_g,\Lambda,\psi$ with $\mathcal H_g^a,\Lambda^a,\psi^a$ respectively to obtain the following analogue of Proposition \ref{prop:mhg} valid for $1\le r\le 2d-2$ (note that $\deg g=d-1$ so we need to replace $d$ with $d-1$):
$$
M^r_{\mathcal H_g^a}=O\left(\frac {q^{2d-2}}{\varphi(g)^2}q^{(1/p-1/2)r}+\frac {r\tau(g)q^{d-1}}{\varphi(g)}q^{(1-1/p)(r-2d+2)+\min(2p,r/p)}\right).
$$
Applying the arguments in the proof of Proposition \ref{prop:masdg} we see that the same bound holds for the shifted family $\{f+b:f\in H_g^a\}$ ($b\in\F_q$) and now from (\ref{eq:union_hga}) (and combining with Proposition (\ref{prop:masdg})) we obtain the following

\begin{prop}\label{prop:masdg1}
Let $g$ be a squarefree polynomial of degree $d$ or $d-1$ and $1\le r\le 2d-2$. Then
$$
M^r_{\AS_{d,g}^\ord}=O\left(\frac {q^{2\deg g}}{\varphi(g)^2}q^{(1/p-1/2)r}+\frac {r\tau(g)q^{\deg g}}{\varphi(g)}q^{(1-1/p)(r-2d+2)+\min(2p,r/p)}\right),
$$
with $\varphi,\tau$ as in Proposition \ref{prop:mhg}.
\end{prop}

\subsection{1-level density for $\AS^\ord_{d,g}$}

In the present subsection we estimate the 1-level density for the family $\AS^\ord_{d,g}$, completing the proof of Theorem \ref{thm:main3}. Throughout the rest of the section $d$ is a natural number and $g\in\F_q[x]$ is monic squarefree with $\deg g\in\{d,d-1\}$. We also fix a test function $\Phi\in\mathcal S(\R)$ with $\mathrm{supp}(\Phi)\subset(-1+\delta,1-\delta)$, for a fixed $\delta>0$. In the present section the asymptotic $O$-notation may depend on $\Phi,\delta$, i.e $O(\cdot)=O_{\Phi,\delta}(\cdot)$, but not on any other parameters. 

To shorten notation we denote 
$M^r=M_{\AS^\ord_{d,g}}^r.$ Note that $M^{-r}=\overline{M^r}$. For $f\in\AS^\ord_{d,g}$ we have $\g=\g(C_f)=(d-1)(p-1)$ and each $L(u,f,\psi)$ has $2d-2$ zeros.
First we write as in (\ref{eqn_FourierExpnansion}):
\begin{multline}\label{eq:fourier_ord}
\mean{\Wone{f}}_{f\in\AS^\ord_{d,g}}=\hat{\Phi}(0)+\frac{1}{2d-2}\sum_{r=1}^\infty \left(
\hat{\Phi}\left(\frac{r}{2d-2}\right)M^r+
\hat{\Phi}\left(\frac{-r}{2d-2} \right)\overline{M^r}
\right)\\=\hat{\Phi}(0)+\frac{1}{2d-2}\sum_{r=1}^{(1-\delta)(2d-2)} \left(
\hat{\Phi}\left(\frac{r}{2d-2}\right)M^r+
\hat{\Phi}\left(\frac{-r}{2d-2} \right)\overline{M^r}
\right).
\end{multline}
We want to apply the bound in Proposition \ref{prop:masdg1}, but we first need to bound the quantities $q^{\deg g}/\varphi(g)$ and $\tau(g)$ ($\varphi,\tau$ are the Euler totient and divisor functions respectively).

\begin{lem}\label{lem:bound_phi_tau}Let $g\in\F_q[x]$ be squarefree of degree $d$ and $\epsilon>0$ a fixed constant. Then
\begin{enumerate}
\item[(i)]$\frac{q^d}{\varphi(g)}=O_\epsilon(d^\epsilon).$
\item[(ii)]$\tau(g)=O_\epsilon(q^{\epsilon d}).$
\end{enumerate}
Here the implicit constants depend only on $\epsilon$ (not on $q,d$).
\end{lem}

\begin{proof} Over $\Z$ the analogous bounds $\frac n{\varphi(n)}=O_\epsilon(\log^\epsilon n),\tau(n)=n^\epsilon$ are well-known consequences of the Prime Number Theorem. See \cite{MoVa06}*{Theorems 2.9,2.11} for even stronger statements. The proofs can be easily adapted to $\F_q[x]$ (we omit the details for brevity).\end{proof}

Now from (\ref{eq:fourier_ord}), Proposition \ref{prop:masdg1} and Lemma \ref{lem:bound_phi_tau} we have (recall that $\Phi,\delta>0$ are fixed and we also use $\min(r/p,2p)\le\sqrt{2r}\le\sqrt d$)
\begin{multline*}\mean{\Wone{f}}_{f\in\AS^\ord_{d,g}}-\hat\Phi(0)\ll\frac 1d\sum_{r=1}^{(1-\delta)(2d-2)}|M^r|
 \\ \le 
\frac 1d\left(\frac {q^{\deg g}}{\phi(g)}\right)^2\cdot\sum_{r=1}^{(1-\delta)(2d-2)}\left(q^{(1/p-1/2)r}+d\tau(d)q^{(1-1/p)(r-2d+2)+\min(2p,r/p)}\right)
 \\ \ll_\epsilon
d^{\epsilon-1}\left(\sum_{r=1}^{(1-\delta)(2d-2)}q^{(1/p-1/2)r}\right)+d^{\epsilon}q^{\epsilon d}
q^{(2/p-2)\delta d+2\sqrt d}\to 0\ll_\epsilon d^{\epsilon-1}+d^\epsilon q^{(\epsilon+(2/p-2)\delta)d+2\sqrt d}\to 0
\end{multline*}
for $d\to\infty$ as long as $\epsilon<\delta$, uniformly in $q,g$. This concludes the proof of Theorem \ref{thm:main3}.

\bibliography{../Tex/mybib}
\bibliographystyle{amsrefs}

\end{document}